\documentclass[reqno, 11pt]{amsart}
\usepackage{amsmath, amsfonts, amsthm, amssymb, stmaryrd, color, cite, mathrsfs}

\usepackage{hyperref}

\newtheorem{theorem}{Theorem}[section]
\newtheorem{lemma}{Lemma}[section]
\newtheorem{proposition}{Proposition}[section]
\newtheorem{corollary}{Corollary}[section]

\theoremstyle{definition}
\newtheorem{definition}{Definition}[section]

\theoremstyle{remark}

\numberwithin{equation}{section}
\allowdisplaybreaks

\newcommand{\R}{{\mathbb R}}

\def\f{\frac}

\def\hf1{^\f{1}{1-\xi^2}}

\def\be{\begin{equation}}
\def\ee{\end{equation}}
\def\bs{\begin{split}}
\def\es{\end{split}}
\def\ba{\begin{align}}
\def\ea{\end{align}}

\title[stochastic Landau-Lifshitz-Bloch equation]
{
stability of invariant measures of the
stochastic 
 Landau-Lifshitz-Bloch equation 
 with vanishing noise}

\author[Z. Qiu]{Zhaoyang Qiu}
\address{School of Applied Mathematics, Nanjing University of Finance and Economics, Nanjing, 210046, China.}
\email{zhqmath@163.com}

\author[D. Huang]{Daiwen Huang}
\address{Institute of Applied Physics and Computational Mathematics, Beijing 100088,  China}
\email{huang\_daiwen@iapcm.ac.cn}

\author[B. Wang]{Bixiang Wang}
\address{Department of Mathematics, New Mexico Institute of Mining and Technology, Socorro, NM 87801, USA.}
\email{bwang@nmt.edu.}

\keywords{stochastic Landau-Lifshitz-Bloch equation, 2D unbounded domains, well-posedness, invariant measure, limiting behavior, tail-ends estimates, artificial viscosity}
\subjclass[2020]{35Q35, 76D05, 35R60, 60F10}

\date{\today}

\begin{document}
\baselineskip=1.2\baselineskip
\begin{abstract}
In this paper, 
we investigate the limiting dynamics of
invariant measures
of
the  
stochastic 
Landau-Lifshitz-Bloch equation 
driven by  the Stratonovich noise
defined on the entire space $\R^2$.
We first prove the set of all invariant measures
of the stochastic equation for small noise is tight
in $H^1(\R^2)$,  and  then prove every limit
of a sequence of invariant measures
of the stochastic equation
must be an invariant measure of the  
limiting system
as  the noise intensity approaches zero.
The main difficulty of the paper is to establish
the tightness of solutions which is caused by the
low regularity of solutions and the non-compactness
of Sobolev embeddings  on   unbounded domains.
To solve the problem, we  first consider a
family of  higher-order
perturbed  viscous  systems  and then use the regularity
as well as the uniform  tail-ends estimates
of  the perturbed solutions to
establish the tightness of solutions
of the original equation by a  limiting
process.
 \end{abstract}

\maketitle
 \section{Introduction}

{\subsection{Physical Backgrounds} In ferromagnetic materials, the Landau-Lifshitz-Gilbert (LLG) equation and the Landau-Lifshitz-Bloch (LLB) equation are two fundamental models describing the dynamics of magnetization. The LLG equation governs the precessional motion and damping of a unit magnetization vector under an effective magnetic field, assuming that the modulus of the magnetization remains constant throughout the evolution, see  \cite{B5, gok}. It is widely used in micromagnetic simulations at low temperatures. However, this assumption becomes invalid at elevated temperatures, especially near or above the Curie temperature, where thermal fluctuations induce variations in the magnetization amplitude.
In order to describe the time evolution of magnetic dynamics near or above the Curie temperature  when the thermal fluctuations significantly affect the spin alignment,  the 
 LLB equation 
 is used to take
  place of the LLG equation. Originally derived by Garanin \cite{gar} from a classical Fokker-Planck formalism for spin ensembles, the LLB equation serves as a thermodynamically consistent extension of the well-known LLG equation. Unlike the LLG model, which preserves the modulus of the magnetization and is valid only at low temperatures, the LLB equation incorporates both transverse and longitudinal relaxation mechanisms, thereby capturing the temperature-induced changes in the magnitude of the magnetization vector near or above the Curie temperature,  see  also \cite{A1,A2,una,ga} for more physical backgrounds.} This equation reads as
$$\frac{\partial \mathbf{u}}{\partial t}
=\gamma \mathbf{u}\times \mathrm{H}_{e\!f\!f}
- \frac{\gamma\alpha_1}{|\mathbf{u}|^2}\mathbf{u}\times (\mathbf{u}\times \mathrm{H}_{e\!f\!f})
+ \frac{\gamma\alpha_2}{|\mathbf{u}|^2}(\mathbf{u}\cdot\mathrm{H}_{e\!f\!f})\mathbf{u},$$
where  $\mathbf{u}=({u}_1, {u}_2, {u}_3)$ is the average spin polarization, $\gamma>0$ represents the gyromagnetic ratio. The notation $\times$ denotes the vector cross product,
and $|\cdot|$ is the Euclidean norm of $\mathbb{R}^3$. $\alpha_1, \alpha_2>0$ are damping parameters which
govern transverse damping (precession-like behavior) and longitudinal relaxation (magnetization length changes) respectively. $\mathrm{H}_{e\!f\!f}$  is  the effective field with the form
$$\mathrm{H}_{e\!f\!f}=\Delta\mathbf{u}-\frac{1}{\alpha_2}(1+\kappa|\mathbf{u}|^2)\mathbf{u},$$
where 
$\kappa>0$ is a constant related to the Curie temperature, which includes contributions from the external field, exchange field, anisotropy field, and thermal fluctuations.
Without loss of generality,
 we will assume $\kappa=\gamma=\alpha_1=\alpha_2=1$.
Then the LLB equation can be written more precisely as
\begin{align}\label{e3}\frac{\partial \mathbf{u}}{\partial t}= \mathbf{u}\times \mathrm{H}_{e\!f\!f}+ \mathrm{H}_{e\!f\!f},\end{align}
see \cite{B1, W1} for detailed calculations.  At high temperatures, the study of
the dispersion of individual trajectories is important.
 For example, when the magnetization is quenched it should describe
the loss of magnetization correlations in different sites of the
sample. In the laser-induced dynamic this is responsible for
the slowing down of the magnetization recovery at high laser
fluency as the system temperature decreases. Therefore,
under these circumstances an application of the stochastic version is significant, which was studied originally by N\'{e}el \cite{N1}.
Since the effective field $\mathrm{H}_{e\!f\!f}$ includes the external force,
  the random fluctuation\
  is introduced
   into the equation by perturbing the effective field  $\mathrm{H}_{e\!f\!f}$ using the Gaussian noise, see \cite{eva,ga1}. Thus, $\mathrm{H}_{e\!f\!f}+
   \frac{\circ dW}{dt}$ takes place of $\mathrm{H}_{e\!f\!f}$
 with $W =    \varepsilon
 \sum_{k=1}^\infty \mathbf{f}_kd W_k$,\
where $    \varepsilon>0$ is the noise intensity,
 $\{W_k\}_{k\geq 1}$ is a sequence  of independent
 standard  real-valued Wiener processes defined in a  complete
   filtered probability space $(\Omega,\mathcal{F}, \{\mathcal{F}_t\}_{t\geq 0}, \mathrm{P})$,
   and 
   $\{\mathbf{f}_k\}_{k= 1}^\infty$ is a sequence of functions satisfying
\begin{align}\label{1.1}
\sum_{k =1}^\infty
\|\mathbf{f}_k\| _{W^{1,\infty}(\mathbb{R}^2, ~\mathbb{R}^3
)\cap H^1(\mathbb{R}^2, ~\mathbb{R}^3)} <\infty.
\end{align}
The symbol 
$\circ dW$ means that  the noise 
is given in the sense of
the Stratonovich differential.
  Then  the stochastic version of 
  the LLB equation   on  $\mathbb{R}^2$ 
  is given by
\begin{equation}\label{1.3}
d\mathbf{u}=\Delta \mathbf{u} dt+ \mathbf{u}\times \Delta \mathbf{u} dt-(1+|\mathbf{u}|^{2})\mathbf{u} dt+\varepsilon\sum_{k=1}^\infty(\mathbf{u} \times \mathbf{f}_k+\mathbf{f}_k)\circ d W_k,
\end{equation}
with initial data $\mathbf{u}_0$.
Note that the  Stratonovich
stochastic equation \eqref{1.3}
is equivalent to the following It\^{o}
stochastic equation  on  $\mathbb{R}^2$:
\begin{equation}\label{wang1}
d\mathbf{u}=\Delta \mathbf{u} dt+ \mathbf{u}\times \Delta \mathbf{u} dt-(1+|\mathbf{u}|^{2})\mathbf{u} dt
+
\frac{\varepsilon^2}{2}
\sum_{k=1}^\infty(\mathbf{u} \times \mathbf{f}_k)\times \mathbf{f}_kdt
+\varepsilon\sum_{k=1}^\infty(\mathbf{u} \times \mathbf{f}_k+\mathbf{f}_k) d W_k.
\end{equation}

 In this paper, we 
 will investigate
 the stability of the dynamics of the
 stochastic LLB equation  with vanishing 
 noise; that is,   
  the limiting behavior of invariant measures of
  \eqref{1.3}  (or \eqref{wang1})   as $\varepsilon
  \to 0$. 
  Without loss of generality, 
  we 
  assume $\varepsilon
  \in [0,1]$ from now on. 
  For the exisence and
  the  limiting behavior
  of invariant measures  of SPDEs, the reader is referred to \cite{bmo, car2, car3, bx3, cb, bx2,kuk1,k1,k3, JH,F1,m1,b, car} and the references therein.

\subsection{Main result}
In this subsection, we present the main
result of the paper on the limiting behavior
of invariant measures of \eqref{1.3}.
To that end, we first recall the
existence and uniqueness of solutions
of  \eqref{1.3} from \cite{Q1}:  for every
    $\mathbf{u}_0\in H^1(\mathbb{R}^2)$
    and  $\varepsilon\in [0,1]$,
    the stochastic  equation \eqref{1.3} admits a unique     solution $\mathbf{u}^\varepsilon$ 
    which is a  
weakly continuous,   $H^1(\mathbb{R}^2)$-valued,
  $\mathcal{F}_t$-adapted process such that
for all  $T>0$ and $p\ge 1$,
\begin{align}\label{wang2}
\mathbf{u}^\varepsilon
\in L^p(\Omega; L^\infty(0,T; H^1(\mathbb{R}^2))\cap L^2(0,T;H^2(\mathbb{R}^2) )), 
    \end{align}
    and
    \be\label{wang3}
    \mathbf{u}^\varepsilon \times
    \Delta \mathbf{u}^\varepsilon
    \in
L^{p}(\Omega; L^r(0,T; L^2
(\mathbb{R}^2
))),
    \quad
    \forall   \ r\in[1,2).
    \ee
Furthermore, it follows from
  \cite{Q1} that
  for every $\varepsilon>0$,
  the stochastic equation \eqref{1.3}
  has at least one 
  invariant measure $\mu^\varepsilon$
  on $H^1
  (\R^2)$.
  In the sequel, 
 we denote by   $\mathcal{I}^\varepsilon$
 the set  of all   invariant measures of \eqref{1.3}
 in $H^1
  (\R^2)$
 corresponding to  $\varepsilon$.  
 
 The main result of this paper is presented  below,
   which 
   is concerned with  the asymptotic stability of invariant measures of  \eqref{1.3}
  with respect to  $\varepsilon$.
   
 \begin{theorem}\label{th1}
If \eqref{1.1} holds, then:

{\em (i)}.  The set  $\bigcup_{\varepsilon
\in [0,1]} \mathcal{I}^\varepsilon$
is tight in $H^1(\mathbb{R}^2)$. 

{\rm (ii)}.  Every limit of  a sequence of invariant measures
of \eqref{1.3} must be an invariant measure
of the limiting system; that is,
if  $\mu$ is a probability
measure on $H^1
(\R^2)$,
$\varepsilon_n \to \varepsilon_0\in [0,1]$
 and
$\mu^{\varepsilon_n}
\in \mathcal{I}^{\varepsilon_n}$ such that
$
\mu^{\varepsilon_{n}}
\to
\mu $ weakly, then  
$\mu \in  \mathcal{I}^{\varepsilon_0}$.
 \end{theorem}
 
 As an immediate consequence of Theorem
 \ref{th1} we find that
 if $\mu^{\varepsilon_n}$
 is a sequence of invariant measures
 of the stochastic equation
  \eqref{1.3} with $\varepsilon_n
 \to 0$, then there exists an invariant
 measure  $ \mu^0$
 of the deterministic equation 
 for  $\varepsilon =0$
 such that, up to a subsequence,
 $\mu^{\varepsilon_n}
 \to \mu^0$ weakly.
 This implies the stability
 of invariant measures of the
 stochastic equation \eqref{1.3}
 with vanishing noise.

We will first prove   the set  $\bigcup_{\varepsilon
\in [0,1]} \mathcal{I}^\varepsilon$
is tight in $H^1(\mathbb{R}^2)$ for which
we need to find a compact subset
$K$
of $H^1(\mathbb{R}^2)$
such that the  probability
measure $ \mu^\varepsilon
(H^1(\mathbb{R}^2) \setminus
K)$ is uniformly small
for all $\mu^\varepsilon
\in   \mathcal{I}^\varepsilon$
and  ${\varepsilon
\in [0,1]}$.
It seems that 
it is  extremely  difficult
to find  such a compact set $K$ 
in 
$H^1(\mathbb{R}^2)$ 
because
the domain
$\R^2$ is unbounded
and   the Sobolev
embedding 
  $H^\alpha(\mathbb{R}^2)
  \hookrightarrow
    H^\beta(\mathbb{R}^2)$ with $\alpha>\beta$
  is   non-compact.
  A classical method to circumvent the
  non-compactness of  Sobolev embeddings
  on unbounded domains  
  is to use the  idea of uniform tail-ends estimates
  by showing that all the  solutions
  of the   equation are uniformly
  small outside a sufficiently large ball
   as in    \cite{wb}.

   In order to apply the argument of the
   uniform tail-ends estimates
   to  the stochastic equation
   \eqref{1.3} in $H^1(\mathbb{R}^2)$,
   we  first need to establish the energy equation
   for  solutions in $H^1(\mathbb{R}^2)$,  
   and  then use the energy equation
   to derive the uniform tail-ends estimates
   of solutions.
   Based on the regularity of solutions
   given by \eqref{wang2},
   if 
    \be\label{wang4}
    \mathbf{u}^\varepsilon \times
    \Delta \mathbf{u}^\varepsilon
    \in L^2(0,T; L^2
(\mathbb{R}^2)),
\quad \text{P-almost surely},
\ee
then  the energy equation
   for \eqref{1.3} in $H^1(\mathbb{R}^2)$
   follows from 
    \cite[Theorem 3.2]{kry1981},
    and in 
      this case,    $ \mathbf{u}^\varepsilon   \in C([0,T];   H^1(\mathbb{R}^2))$,
     P-almost surely.

Unfortunately, by \eqref{wang3}
we find that 
    $\mathbf{u}^\varepsilon \times
    \Delta \mathbf{u}^\varepsilon
    $  only belongs to 
   $ L^r(0,T; L^2
(\mathbb{R}^2))$ for $r<2$
but not  for $r=2$,
P-almost surely,
 and hence  it does not
 satisfy the condition \eqref{wang4}.
 Consequently,
 the energy equation
 for  solutions of \eqref{1.3}
 in 
  $H^1(\mathbb{R}^2)$
  is unavailable and the uniform tail-ends
  estimates in $H^1(\mathbb{R}^2)$
  cannot be established by the classical method.
  That is why   the uniform
  tail-ends estimates of solutions
  for the two-dimensional
  LLB equation 
  was left as 
  an open problem
   in \cite{Q1},
  where the authors only
  proved the uniform tail-ends estimates
  for the one-dimensional LLB equation
  based on 
  the embedding
  $H^1(\R) \hookrightarrow L^\infty (\R)$.
  
  The first goal of the present paper is
  to solve the open problem for the
  two-dimensional LLB equation
  \eqref{1.3} by employing 
  a \textbf{\emph{new}}
  method
  instead of the classical one.
  Our idea is to perturb the LLB equation
  \eqref{1.3}
  by adding a higher-order viscous term
  $\delta    \Delta^2 \mathbf{u}$
  with $\delta\in (0,1]$
  and derive the   uniform tail-ends estimates
  of  \eqref{1.3}
  by a limiting process as $\delta \to 0$.
  More precisely, we consider the
  following viscous equation on $\R^2$:
\begin{align}\label{1.3*}
&d\mathbf{u}^\delta=\Delta \mathbf{u}^\delta dt-\delta \Delta^2 \mathbf{u}^\delta dt+ \mathbf{u}^\delta\times \Delta \mathbf{u}^\delta dt-(1+|\mathbf{u}^\delta|^{2})\mathbf{u}^\delta dt\nonumber\\
&+\frac{\varepsilon^2}{2}
\sum_{k=1}^\infty(\mathbf{u}^\delta \times \mathbf{f}_k)\times \mathbf{f}_kdt+\varepsilon\sum_{k=1}^\infty(\mathbf{u}^\delta \times \mathbf{f}_k+\mathbf{f}_k)d W_k,
\end{align}
where $\delta\in (0, 1]$ and $\Delta^2$ is the biharmonic operator given by
$\Delta^2=\Delta\circ\Delta
$.
 
 We first prove that for every
 $ \mathbf{u}_0 
 \in H^1(\R^2)$ and $\delta\in  (0,1]$,
 the  viscous equation
 \eqref{1.3*} has a unique solution
$ \mathbf{u}^\delta$ which satisfies
    \eqref{wang2}  and 
  the further regularity:
 \be\label{wang5}
 \mathbf{u}^\delta\times \Delta \mathbf{u}^\delta \in
L^{p}(\Omega; L^2(0,T; L^2
(\mathbb{R}^2
))),
    \quad
    \forall \ p\ge1,
    \ee
see Proposition \ref{pro2}. 
By \eqref{wang5} we  then obtain the
  energy  equation
  for the solution  $\mathbf{u}^\delta$
  of \eqref{1.3*} in $H^1(\mathbb{R}^2)$,
  by which  and  a cut-off technique,
  we further derive the     uniform tail-ends estimates of  $\mathbf{u}^\delta$ in $H^1(\mathbb{R}^2)$,  
   see Lemma  \ref{tail_3}.
   We finally consider the 
  the vanishing viscosity limit 
  of  $\mathbf{u}^\delta$
   as $\delta\rightarrow 0$,
  and prove  $\mathbf{u}^\delta$
  converges to the solution of the LLB equation
  \eqref{1.3} in probability.
  This convergence 
  makes it possible to 
  obtain the  uniform tail-ends estimates of 
  \eqref{1.3}  in 
   $H^1(\mathbb{R}^2)$
   by  a limiting process and
   the  tail-ends
   estimates of the  viscous solution
   $\mathbf{u}^\delta$ of  \eqref{1.3*},
   see Proposition  \ref{pro3a} and 
   Proposition
   \ref{pro3}.

The second goal of this  paper
is to show every limit
of a sequence of invariant measures of
\eqref{1.3} in $H^1(\R^2)$  must be invariant, for
which,
by the
   standard method \cite{D1, li2022,kuk},
  the convergence of solutions
   in probability in $H^1(\mathbb{R}^2)$
   with respect to
   $\varepsilon$
   is necessary.
   However, due to the higher-order
   nonlinearity of the term 
   $\mathbf{u}^\varepsilon
   \times \Delta \mathbf{u}^\varepsilon$
   and the low regularity of solutions,
   the
   convergence of solutions of \eqref{1.3}
   in probability in $H^1(\mathbb{R}^2)$
   is unavailable, and hence the standard method
   is invalid in this case.
   
   In order to solve the problem, we first
   prove the convergence of solutions
   of \eqref{1.3}
   in probability in  the space $L^2(\R^2)$
   with respect to $\varepsilon$,
   which is weaker than the 
   convergence in the space 
   $H^1(\mathbb{R}^2)$.
   We then  
 employ a regularization method
 to transform a function in $L^2(\R^2)$
 into  a more regular function in  $H^1(\mathbb{R}^2)$.
 We finally  establish  the invariance of
   the limit of invariant measures of \eqref{1.3}
    by  passing to the limit as 
    the regularization  parameter
     approaches
    zero,  
   see the proof of Theorem \ref{th1}
   in the last section.

\subsection{Current  works  on the LLB equation}
Due to its physical importance, 
the LLB equation  has  been
extensively studied 
  in recent years.  
Using the Faedo-Galerkin    scheme
 and the  compactness method,   Le \cite{le}
 first 
 proved the existence of weak solutions of
  \eqref{e3} in a smooth bounded domain.
 Then Guo et al. \cite{guo} established
 the well-posedness of strong solutions of \eqref{e3} evolving in a finite-dimensional
  closed Riemannian manifold, and    Guo et al. \cite{li} further proved
  the existence of   weak and strong solutions
   to Landau-Lifshitz-Bloch-Maxwell equations with polarization. Wang et al. \cite{whq} considered the limiting behavior of \eqref{e3} with the stray field in the thin film.
   For \eqref{1.3}  in a bounded domain,
  Jiang et al. \cite{W1} and    Brze\'{z}niak et al. \cite{B1}  established  the existence of martingale solutions,
   strong 
  solutions   and  invariant measures.
When the equation
is driven by   transport noise,
the existence and   regularity
of solutions  were  discussed  by
 Qiu et al. \cite{q1}. 
 The reader is  further referred to
  \cite{gok,gok2, Q2, soh} for the large deviations, ergodicity and Wong-Zakai approximations
  of the LLB equation.

  For \eqref{1.3} defined on unbounded domains, the authors of  \cite{Q1} proved   the existence of invariant measures  in the two-dimensional case,
  and discussed 
   the limiting behavior of invariant measures 
   in the one-dimensional case, and
    in  \cite{Q22} the authors proved 
     the existence, uniqueness and upper semi-continuity of  random attractors in the one-dimensional case. The goal of this paper is to
     study the   limiting behavior of invariant measures
     of \eqref{1.3}
     in the two-dimensional case which was left in \cite{Q1}
     as an open problem.

The rest of this paper is organized as follows.
 In  the next section,
  we establish the existence and uniqueness of     solutions of \eqref{1.3*} by the domain expansion method.
  Then we prove the convergence of solutions
  of \eqref{1.3*} as the viscosity approaches zero
  in Section 3, and 
  establish the  uniform tail-ends estimates 
  of solutions in 
  Section 4.
  We  finally prove 
  Theorem \ref{th1} in  the last section.

\section{Well-posedness of the 
stochastic viscous  equation}
In this section we consider the global well-posedness of the viscous  equation \eqref{1.3*}.
We first  introduce the Sobolev spaces
which will be  used in the sequel. 
Let $\mathscr{D}$
be an open set in
$\mathbb{R}^2$.
Given $p\ge 1$, let
$$
L^p(\mathscr{D}, \mathbb{R}^3)
=\left \{
\mathbf{u}: \mathscr{D}
\to \mathbb{R}^3:
\int_{\mathscr{D}}
|\mathbf{u} (x)|^p dx <\infty
\right \} .
$$
The norm of
$
L^p(\mathscr{D}, \mathbb{R}^3)
$ is  denoted by
$\| \cdot \|_{L^p(\mathscr{D}, \mathbb{R}^3) }$.
In particular,
for $p=2$, the inner
product
and the norm of
$
L^2(\mathscr{D}, \mathbb{R}^3)
$
are written as
$(\cdot, \cdot)_{L^2(\mathscr{D}, \mathbb{R}^3)}$
and $\| \cdot \|_{L^2(\mathscr{D}, \mathbb{R}^3)}$,
respectively.
For every   $k\in \mathbb{N}$
and  $p \ge 1$, let
$$
W^{k,p}(\mathscr{D},
\mathbb{R} ^3 )
=\left \{
\mathbf{u}
\in L^p
( \mathscr{D},
  \mathbb{R}^3) :
  \partial ^\alpha
  \mathbf{u}
  \in
  L^p
  ( \mathscr{D},
  \mathbb{R}^3),
  \ |\alpha | \le k
 \right \}.
 $$
 The norm of
 $W^{k,p}(\mathscr{D},
 \mathbb{R} ^3 )$
 is denoted by
 $\|\cdot \|_{W^{k,p}(\mathscr{D},
 	\mathbb{R} ^3 )}$.
 	Let
$W_0^{k,p}(\mathscr{D},
\mathbb{R}^3 )$
be the closure of $C_0^\infty(\mathscr{D}, \mathbb{R}^3)$ with respect to
 the topology of $W^{k,p}(\mathscr{D}, \mathbb{R}^3)$.
   For  $p=2$, we
   write  $H^k(\mathscr{D},
   \mathbb{R}^3)
   =  W^{k,2}(\mathscr{D} ,
   \mathbb{R} ^3
   )$
    which is a Hilbert space
    with inner product $(\cdot,\cdot)_{H^k(
  \mathscr{D},
  \mathbb{R} ^3)}=\sum\limits_{|\alpha|\leq k}(\partial^\alpha \cdot, \partial^\alpha \cdot)_{L^2(\mathscr{D}, \mathbb{R}^3)}$.
  Denote by $H^{-k}(
  \mathscr{D},
  \mathbb{R} ^3)$ the dual space of $H^{k}(
  \mathscr{D},
  \mathbb{R} ^3)$.
     If no confusion occurs, we will use
   $L^p (\mathscr{D} )$
   for  $L^p (\mathscr{D},
\mathbb{R} ^3 )$, and use  similar notations
  for other spaces.

 Let $Y $ be a  separable  Hilbert space
 with inner product $ (\cdot, \cdot )_Y $.
 Denote by
   $Y_w$ the    space
   $Y$
endowed with  weak topology.  Given $T>0$, 
let
$$
C([0,T]; Y_w)=
\left \{
\mathbf{u}: \mathbf{u} \ \text{is  a  continuous  function  from }  [0,T]
\ \text{to}\ Y_w \right \},
$$
which is equipped  with the following topology:
$\mathbf{u}_n \to \mathbf{u}$ in
$C([0,T]; Y_w)$  if for every $y\in Y$,
$\lim\limits_{n\to \infty}
\sup\limits_{t\in [0,T]}
|(\mathbf{u}_n(t) - \mathbf{u}(t), y)_Y|
=0$.
Given $R>0$, let
$C([0,T]; Y^R_w)$ be a subset of
$C([0,T]; Y_w)$ given by
$$
C([0,T]; Y^R_w)=
\left \{
\mathbf{u}\in C([0,T]; Y_w): \sup_{t\in [0,T]}
\| \mathbf{u}(t)\|_Y
\le R    \right \}.
$$
Since the weak topology of
$Y$ on the  closed ball
of radius $R$ centered at the origin
is metrizable, we see that
$C([0,T]; Y^R_w)$ is a complete metric space.

A  solution of the viscous equation
\eqref{1.3*} is understood in the following sense.

\begin{definition}\label{defa}  
 Given 
  $ \mathbf{u}_0 \in H^1(\R^2)$,
  $\varepsilon\in [0,1]$, $\delta\in(0,1]$ and 
  $T>0$,
   a  
 continuous     $H^1(\mathbb{R}^2)$-valued  $\mathcal{F}_t$-adapted
 process $\mathbf{u}^{\varepsilon,\delta}$
   is called a solution of \eqref{1.3*} if 
$$
 \mathbf{u}^{\varepsilon, \delta}\in 
L^2(\Omega; L^\infty(0,T; H^1(\mathbb{R}^2)))
\cap
L^2(\Omega; L^2(0,T; H^3(\mathbb{R}^2)))
 $$ 
such that for all $t\in [0,T]$,   $\mathrm{P}$-almost 
surely,
\begin{align*}
&\mathbf{u}^{\varepsilon, \delta}(t)=\mathbf{u}_0+\int_{0}^{t}\Delta \mathbf{u}^{\varepsilon, \delta}(s)ds-\int_{0}^{t}\delta\Delta^2 \mathbf{u}^{\varepsilon,\delta}(s)ds+\int_{0}^{t}\mathbf{u}^{\varepsilon,\delta}(s)\times \Delta \mathbf{u}^{\varepsilon,\delta}(s)ds\nonumber\\
&\quad-\int_{0}^{t}(1+|\mathbf{u}^{\varepsilon,\delta}(s)|^{2})\mathbf{u}^{\varepsilon,\delta}(s)ds
+\frac{\varepsilon^2}{2}\sum_{k= 1}^\infty
\int_{0}^{t}(\mathbf{u}^{\varepsilon,\delta}(s)\times \mathbf{f}_k)\times \mathbf{f}_kds
+\varepsilon\sum_ {k= 1}^\infty\int_{0}^{t}(\mathbf{u}^{\varepsilon,\delta}(s)\times \mathbf{f}_k+\mathbf{f}_k)d W_k 
\end{align*}
in   
$H^{-1}(\mathbb{R}^2 )$.
\end{definition}

Our main result of this section is formulated as follows.

\begin{proposition}\label{pro2}  
  If \eqref{1.1} holds,  then for every
  $ \mathbf{u}_0 \in H^1(\R^2)$,
      $\varepsilon\in [0,1]$ and $\delta\in(0,1]$, the stochastic viscous equation \eqref{1.3*} admits a unique   solution $\mathbf{u}^{\varepsilon, \delta}$ in the sense of
       Definition \ref{defa}.
       
       Furthermore,  
 for  every $ p\ge 1$, $\varepsilon\in [0,1]$ and $\delta\in(0,1]$, 
   the  following uniform estimates hold:
\be \label{pro2 1}
\mathrm{E}\left(\sup_{t\in[0,T]}
\| \mathbf{u}^{\varepsilon,\delta} (t, \mathbf{u}_0 )\|^{2p}_{H^1 (\R^2) }
\right)
\le   M_1,
\ee
\be\label{pro2 2}
  \mathrm{E}\left(\int_{0}^{T}
\left(\|
\mathbf{u}^{\varepsilon,\delta}
(t, \mathbf{u}_0 )\|^2_{H^2 (\R^2) }+
\delta \| \mathbf{u}^{\varepsilon,\delta}
(t, \mathbf{u}_0 )\|^2_{H^3 (\R^2) } \right)dt
\right)^p
\leq M_1,
\ee
 where $M_1=M_1
 (\mathbf{u}_0, T,  p)>0$ is
 a constant 
 depending on
 $\mathbf{u}_0, T $ and $p$ but not on
    $  \delta$  
    or  $ \varepsilon$. 
    In addition,  
    for all  $\alpha\in (0,\frac{1}{2})$ 
    and $q\ge 2$,  
     we have
  \begin{equation} \label{pro2 3}
  \mathrm{E}
  \left (
  \|  \mathbf{u}^{\varepsilon,\delta} \|^{q}
  _{ W^{\alpha, q}
  	(0,T; H^{-1} (\R^2 ))}
  \right )
  \le M_2,
  \end{equation}
where $M_2=M_2 
 (\mathbf{u}_0, T,  \alpha, q)>0$ is
 a constant 
 depending on
 $\mathbf{u}_0, T, \alpha $ and $q$ but not on
    $ \delta$ 
    or  $ \varepsilon$.            
\end{proposition}

\begin{proof}

The proof is similar to \cite{Q1}, and hence the
details are omitted. We here only sketch the main
idea of the proof.
 For simplicity,   we will
suppress
  the parameter $\varepsilon$
  in  $\mathbf{u}^{\varepsilon,\delta}$
  when no confusion occurs.
  
  The proof consists of two steps: we first
  prove the existence and uniqueness of solutions
  defined in a bounded domain
  $\mathscr{D}_n
=\{x\in \mathbb{R}^2: |x|
<n \}$ for 
every $n\in \mathbb{N}$, and then
obtain  a solution to \eqref{1.3*}
defined in the entire space  $\R^2$
by taking the limit as  $n\to \infty$.

Take  a smooth
function $\theta : \mathbb{R}^2
\to \mathbb{R}$ such that
$0\le \theta (x) \le 1$
for all $x\in \mathbb{R}^2$ and
\begin{equation}\label{cutoff}
\theta (x)=\left\{\begin{array}{ll}
\!\!0, ~ \text{ if  } |x|\geq {\frac 34}; \\\\
\!\!1,~ \text{ if  } |x|\leq \frac{1}{2}.
\end{array}\right.
\end{equation}
 Given $n\in \mathbb{N}$,
let
$\theta_n (x) = \theta  \left ({\frac xn}
\right )$ for all $x\in \mathbb{R}^2$.

For every $n\in \mathbb{N}$,
consider    the  equation
  in    $\mathscr{D}_n$:
\begin{equation}\label{e1*}
\left\{\begin{array}{ll}
\!\!d\mathbf{u}^{\delta,n}=\Delta \mathbf{u}^{\delta,n}dt-\delta \Delta^2 \mathbf{u}^{\delta,n} dt+\mathbf{u}^{\delta,n}\times \Delta \mathbf{u}^{\delta,n}dt\\\\
\!\!\qquad-(1+|\mathbf{u}^{\delta,n}|^{2})\mathbf{u}^{\delta,n}dt+ 
\frac{\varepsilon^2}{2}\sum_{k= 1}^\infty
(\mathbf{u}^{\delta,n}\times \mathbf{f}_k)\times \mathbf{f}_kdt
+\varepsilon\sum_ {k= 1}^\infty(\mathbf{u}^{\delta,n}\times \mathbf{f}_k+\mathbf{f}_k)d W_k,\\\\
\!\!\mathbf{u}^{\delta,n}|_{\partial\mathscr{D}_n}=\Delta\mathbf{u}^{\delta,n}|_{\partial\mathscr{D}_n}=0,\\\\
\!\!\mathbf{u}^{\delta,n}(0,x)=\theta_n
(x)\mathbf{u}_0(x), \quad x\in \mathscr{D}_n,
\end{array}\right.
\end{equation}
 where $\mathbf{u}_0
\in H^1(\mathbb{R}^2)$.

Since the domain 
$\mathscr{D}_n$ is bounded,
by the 
Faedo-Galerkin method and the arguments
of \cite{B1, W1}, we infer that for every
 $\mathbf{u}_0
\in H^1(\mathbb{R}^2)$,
the stochastic equation
\eqref{e1*} has 
a unique      solution $\mathbf{u}^{\delta,n}$ in the   sense that
   $\mathbf{u}^{\delta,n}$ is a 
continuous,   $H_0^1(\mathscr{D}_n)$-valued,  $\mathcal{F}_t$-adapted process  such that
\begin{align}\label{1.3c}
\mathbf{u}^{\delta,n}\in L^p(\Omega; L^\infty(0,T; H_0^1(\mathscr{D}_n
    ))\cap L^2(0,T;H^3(\mathscr{D}_n
   ) )),  \quad \forall \ p\ge 1,
    \end{align}
    \be\label{1.3c1}
    \mathbf{u}^{\delta,n} \times
    \Delta \mathbf{u}^{\delta,n}
    \in
L^{p}(\Omega; L^2(0,T; L^2
(\mathscr{D}_n
   ))),
    \quad
    \forall \ p\ge1,
    \ee
and   $\mathbf{u}^{\delta,n}$
satisfies \eqref{e1*}   in $H^{-1}(\mathscr{D}_n
)$, $\mathrm{P}$-almost surely.

Moreover, by It\^{o}'s formula, after detailed
calculations,  we  obtain 
the  following uniform estimates:
 for  every $n\in \mathbb{N}$, $p\in [1,\infty)$, $\varepsilon\in [0,1]$ and $\delta\in(0,1]$, 
\be \label{e0}
\mathrm{E}\left(\sup_{t\in[0,T]}
\|\mathbf{u}^{\delta,n}(t, \mathbf{u}_0 )\|^{2p}_{H_0^1 (\mathscr{D}_n) }
\right)
\le c_1,
\ee
\be\label{es0}
  \mathrm{E}\left(\int_{0}^{T}
\left(\|\mathbf{u}^{\delta,n}(t, \mathbf{u}_0 )\|^2_{H^2 (\mathscr{D}_n) }+
\delta \|\mathbf{u}^{\delta,n}(t, \mathbf{u}_0 )\|^2_{H^3 (\mathscr{D}_n) } \right)dt
\right)^p
\leq c_1,
\ee
 where $c_1=c_1
 (\mathbf{u}_0, T,  p)>0$ is
 a constant 
 depending on
 $\mathbf{u}_0, T $ and $p$ but not on
    $ n, \delta$  
    or  $ \varepsilon$. 
    In addition,  
    for all  $\alpha\in (0,\frac{1}{2})$ 
    and $q\ge 2$,  
     we have
  \begin{equation} \label{1.50b}
  \mathrm{E}
  \left (
  \| \mathbf{u}^{\delta,n}\|^{q}
  _{ W^{\alpha, q}
  	(0,T; H^{-1} (\mathscr{D}_n) )}
  \right )
  \le c_2,
  \end{equation}
where $c_2=c_2 
 (\mathbf{u}_0, T,  \alpha, q)>0$ is
 a constant 
 depending on
 $\mathbf{u}_0, T, \alpha $ and $q$ but not on
    $ n, \delta$ 
    or  $ \varepsilon$.

Note that the solution
$\mathbf{u}^{\delta, n}$ is defined
in $\mathscr{D}_n$ only.
To obtain a solution to \eqref{1.3*},
we need to  extend 
$\mathbf{u}^{\delta, n}$
from  
$\mathscr{D}_n$ to 
the entire space $\R^2$.
Let  $\mathbf{v}^{\delta, n}
(t,x) =\theta_n (x) \mathbf{u}^{\delta,n}
(t,x) $ for all $ t\in [0,T],
x\in \mathscr{D}_n$, $n\in \mathbb{N}$ and  $\delta\in (0,1]$. Then
  set
 $\mathbf{v}^{\delta, n}(t,x) =0$
 for $ t\in [0,T] $ and $x\in
  \mathbb{R}^2\setminus\mathscr{D}_n$.
  The function 
   $\mathbf{v}^{\delta, n}$ 
   is defined on $[0,T] \times \R^2$
   and satisfies the
      uniform
   estimates given by
   \eqref{e0}-\eqref{1.50b}.
   
   Following the argument of \cite{Q1},
   by taking the limit
   of the sequence $\{\mathbf{v}^{\delta, n}
   \}_{n=1}^\infty$ as $n \to \infty$,
   we can  obtain a martingale solution
   to \eqref{1.3*}. Moreover,
   one can verify that the solutions
   of \eqref{1.3*}
   have the  pathwise uniqueness property
   as in \cite{Q1}, which  further implies the
   existence and uniqueness of strong solutions
   to \eqref{1.3*}.
   Finally,  by \eqref{e0}-\eqref{1.50b}
   we obtain 
    \eqref{pro2 1}-\eqref{pro2 3},
    which completes the proof. 
    \end{proof}

 In the next section, we will
 investigate the convergence
 of solutions of the viscous equation \eqref{1.3*}
 as the viscosity approaches zero,
 which is crucial to establish the
 uniform tail-ends 
 estimates of solutions and investigate the stability
 of invariant measures of the LLB equation \eqref{1.3}.

\section{The vanishing viscosity limit}
In this section, we consider the vanishing 
viscosity limit of the
stochastic viscous equation
\eqref{1.3*}  and  prove the following convergence
result.

 \begin{proposition}\label{pro3.1}
If   \eqref{1.1} holds and
$\mathbf{u}_0
 \in H^1
 (\mathbb{R}^2)$, then the solutions
 $\mathbf{u}^{\varepsilon,\delta}  $
 and $  \mathbf{u}^{\varepsilon}$   of
   \eqref{1.3*} and  \eqref{1.3}
   satisfy: for every $\varepsilon \in [0,1]$,
  $$
\lim_{\delta
\to 0}
\mathbf{u}^{\varepsilon, 
\delta}(\cdot, \mathbf{u}_0)=
\mathbf{u}^{\varepsilon} (\cdot, \mathbf{u}_0)
 \quad \text{in probability in } \ 
  C([0,T]; L^{2}(\mathbb{R}^2))\cap L^2(0,T;H^1
 (\mathbb{R}^2)) .
 $$
 \end{proposition}

In order to prove Proposition
 \ref{pro3.1}, we
 need 
 the following   uniform tail-ends estimates of solutions   of \eqref{1.3*}
 to overcome
 non-compactness of
 Sobolev embeddings
 on unbounded domains.

 \begin{lemma}\label{tail 6}
 If  \eqref{1.1} holds, then for every
 $\epsilon'>0$,
 $T>0$  and
 $\mathbf{u}_0
 \in H^1
 (\mathbb{R}^2)$, 
 there exists $J=J(\epsilon', T, \mathbf{u}_0)
 \ge 1$ such that for all
 $j\ge J$,  $\varepsilon
 \in [0,1]$ and
 $ \delta\in (0,1]$, the solution $\mathbf{u}^{\varepsilon, \delta}(t,
\mathbf{u}_0 )$
 of \eqref{1.3*} satisfies:
$$
 \mathrm{E}
\left ( \sup_{t\in [0,T]}
 \int_{ |x|>j }
 |
 {\mathbf{u}}^{\varepsilon, \delta}(t,
\mathbf{u}_0 )(x) |^2 dx
\right )  <\epsilon'.
$$
\end{lemma}

\begin{proof}
 Given $j\in \mathbb{N}$, denote by
   $\phi_j(x)=\phi(\frac{x}{j})$, where $\phi (x) =1-\theta (x)$ for all
$x\in \mathbb{R}^2$, and $\theta$ is
the smooth function given by \eqref{cutoff}.
By
     It\^{o}'s formula,  we  get
     from \eqref{1.3*} that
 for all $r, t\in [0,T]$
     with $r\le t$,  
  $$ \|\phi_j {\mathbf{u}}^{\varepsilon, \delta}
(r)
\|_{L^2
(\R^2)   }^2
=
\|\phi_j\mathbf{ u}_0
\|_{L^2
(\R^2) }^2  
 -2\int_0^r ( \delta\Delta^2 \mathbf{ u}^{
 \varepsilon, \delta}
(s), \phi^2_j\mathbf{ u}^{ \varepsilon, \delta}(s ) )_{(H^{-1}
(\R^2),   H^{1}
(\R^2))  } ds
$$
$$+2\int_0^r ( \Delta \mathbf{ u}^{ \varepsilon, \delta}
(s), \phi^2_j\mathbf{ u}^{ \varepsilon, \delta}(s ) )_{L^2
(\R^2) } ds
+2
\int_0^r
( \mathbf{ u}^{ \varepsilon, \delta}(s)\times \Delta \mathbf{ u}^{ \varepsilon, \delta}
(s), \phi^2_j\mathbf{ u}^{ \varepsilon, \delta}(s) )_{L^2
(\R^2) } ds$$
$$-2
\int_0^r
((1+|\mathbf{ u}^{ \varepsilon, \delta }(s)|^{2})\mathbf{u}^{ \varepsilon, \delta}(s),
\phi^2 _j\mathbf{ u}^{ \varepsilon, \delta}(s) )_{L^2
(\R^2) }ds
+2\varepsilon\sum_{k=1}^\infty
\int_0^r
(\phi_j \mathbf{f}_k  , \phi_j\mathbf{ u}^{ \varepsilon, \delta}(s) )_{L^2
(\R^2 )}
dW_k
$$
\be\label{tail_1 p1}
+\varepsilon^2\sum_{k=1}^\infty
\int_0^r
\|\phi_j (\mathbf{ u}^{ \varepsilon, \delta}(s) \times \mathbf{f}_k+\mathbf{f}_k)\|_{L^2
(\R^2) }^2ds
+\varepsilon^2 
\sum_{k=1}^\infty\int_0^r
(
(\mathbf{ u}^{ \varepsilon, \delta}(s) \times \mathbf{f}_k)\times \mathbf{f}_k, \phi^2_j\mathbf{u}^{ \varepsilon, \delta}(s))_{L^2
(\R^2) }ds.
\ee
For the artificial viscosity term, we have
$$
-2\int_0^r ( \delta\Delta^2 \mathbf{ u}^{
 \varepsilon, \delta}
(s), \phi^2_j\mathbf{ u}^{ \varepsilon, \delta}(s ) )_{(H^{-1}
(\R^2),   H^{1}
(\R^2))  } ds 
$$
$$
=2\int_0^r ( \delta\nabla\Delta \mathbf{ u}^{ \varepsilon, \delta}
(s), \nabla(\phi^2_j\mathbf{ u}^{ \varepsilon, \delta}(s )))_{L^{2}
(\R^2 ) } ds
$$
$$
=\frac{4}{j}\int_0^r ( \delta\nabla\Delta \mathbf{ u}^{ \varepsilon, \delta}
(s), \phi_j
\nabla \phi({x/j})\mathbf{ u}^{ \varepsilon, \delta}(s))_{L^{2}
(\R^2) } ds+2\int_0^r ( \delta\nabla\Delta \mathbf{ u}^{ \varepsilon, \delta}
(s), \phi^2_j\nabla\mathbf{ u}^{ \varepsilon, \delta}(s ))_{L^{2}
(\R^2) } ds
$$
$$
=\frac{4}{j}\int_0^r \delta( \nabla\Delta \mathbf{ u}^{ \varepsilon, \delta}
(s), \phi_j
\nabla \phi({x/j})\mathbf{ u}^{ \varepsilon, \delta}(s))_{L^{2}
(\R^2) } ds-2\int_0^r \delta(\Delta \mathbf{ u}^{ \varepsilon, \delta}
(s), \phi^2_j\Delta\mathbf{ u}^{ \varepsilon, \delta}(s))_{L^{2}
(\R^2) } ds
$$
$$
- {\frac 4j}
\int_0^r \delta (\Delta \mathbf{ u}^{ \varepsilon, \delta}
(s), \phi_j \nabla\mathbf{ u}^{ \varepsilon, \delta}(s )
 \nabla\phi (x/j)
)_{L^{2}
(\R^2) }  ds
$$
$$
\le \frac{4}{j}\int_0^r \delta( \nabla\Delta
 \mathbf{ u}^{ \varepsilon, \delta}
(s), \phi_j
\nabla \phi({x/j})\mathbf{ u}^{ \varepsilon, \delta}(s))_{L^{2}
(\R^2) } ds
 - {\frac 4j}
\int_0^r \delta (\Delta \mathbf{ u}^{ \varepsilon, \delta}
(s), \phi_j \nabla\mathbf{ u}^{ \varepsilon, \delta}(s )
 \nabla\phi (x/j)
)_{L^{2}
(\R^2) }  ds
$$
 \begin{align}\label{tail_1 p}
\leq \frac{c_1}{j}\int_0^r \delta\|\mathbf{ u}^{ \varepsilon, \delta}
(s)\|^2_{H^{3}
(\R^2) } ds,
\end{align}
where $c_1>0$ is
a constant  independent of $j,  \delta$
and $ \varepsilon$.

The remaining terms on the right-hand side of \eqref{tail_1 p1} can be
estimated  as \cite[(2.9)-(2.12)]{Q1}
for which we have 
$$2\int_0^r ( \Delta \mathbf{ u}^{ \varepsilon, \delta}
(s), \phi^2_j\mathbf{ u}^{ \varepsilon, \delta}(s ) )_{L^2
(\R^2) } ds
+2
\int_0^r
( \mathbf{ u}^{ \varepsilon, \delta}(s)\times \Delta \mathbf{ u}^{ \varepsilon, \delta}
(s), \phi^2_j\mathbf{ u}^{ \varepsilon, \delta}(s) )_{L^2
(\R^2) } ds
$$
$$
+\varepsilon^2\sum_{k=1}^\infty
\int_0^r
\|\phi_j (\mathbf{ u}^{ \varepsilon, \delta}(s) \times \mathbf{f}_k+\mathbf{f}_k)\|_{L^2
(\R^2) }^2ds
+\varepsilon^2 \sum_{k=1}^\infty\int_0^r
(
(\mathbf{ u}^{ \varepsilon, \delta}(s) \times \mathbf{f}_k)\times \mathbf{f}_k, \phi^2_j\mathbf{u}^{ \varepsilon, \delta}(s))_{L^2
(\R^2) }ds
$$
$$
-2
\int_0^r
((1+|\mathbf{ u}^{ \varepsilon, \delta}(s)|^{2})\mathbf{u}^{ \varepsilon, \delta}(s),
\phi^2 _j\mathbf{ u}^{ \varepsilon, \delta}(s) )_{L^2
(\R^2) }ds
$$
\begin{align}
\leq \frac{c_2}{j}
\int_0^r
\|\mathbf{u}^{ \varepsilon, \delta} (s) \|_{H^1
 (\R^2) }^2 ds+3 \sum_{k=1}^\infty\|\mathbf{f}_k\|^2_{L^\infty
(\mathbb{R} ^2) }
\int_0^r \|\phi_j
\mathbf{u}^{ \varepsilon, \delta}(s) \|_{L^2
(\R^2)
}^2 ds
+ 2r \sum_{k=1}^\infty\|\phi_j \mathbf{f}_k\|
_{L^2(\R^2) }^2,
\end{align}
and
$$
2
\varepsilon\mathrm{E}
\left (\sup_{0\le r\le t}
\left |\int_0^r
\sum_{k=1}^\infty
 (\phi_j \mathbf{f}_k  , \phi_j\mathbf{u}^{ \varepsilon, \delta}(s) )_{L^2
(\R^2) }
dW_k \right |
\right )
$$
\be\label{tail_1 p3a}
\le
c_3 \sum_{k=1}^\infty
 \| \phi_j \mathbf{f}_k\|^2_{L^2(\mathbb{R}^2) }
 + c_3
 \int_0^t
 \mathrm{E}
\left ( \| \phi_j\mathbf{u}^{ \varepsilon, \delta}(s)
\|^2 _{L^2
(\R^2) }
\right )
ds,
\ee
where $c_2$ and $ c_3 $ are
positive numbers
 independent of $j,\delta$
 and $\varepsilon$.

It follows from 
  \eqref{tail_1 p1}-\eqref{tail_1 p3a}
  that
for all $t\in [0,T]$ and $j \in \mathbb{N}$, $\varepsilon\in [0,1]$ and
 $\delta\in (0,1]$,
$$
 \mathrm{E}
 \left (\sup_{0\le r \le t}
  \|\phi_j \mathbf{u}^{ \varepsilon, \delta} (r)
 \|_{L^2(\R^2)
 }^2
 \right )
 \leq
 \mathrm{E}
 \left (
 \|\phi_j \mathbf{u}_0
 \|_{L^2(\R^2)
 }^2
 \right )
 $$
 $$
 +\left (c_3
 +
  3\sum_{k=1}^\infty\|\mathbf{f}_k\|^2_{L^\infty
(\mathbb{R} ^2) }
\right )
\int_0^t
 \mathrm{E}
 \left (
 \|\phi_j
\mathbf{u}^{ \varepsilon, \delta} (s) \|_{L^2
(\R^2)
}^2
\right ) ds + (2T + c_3) \sum_{k=1}^\infty\|\phi_j \mathbf{f}_k\|
_{L^2(\mathbb{R}^2)  }^2
$$
\begin{equation}\label{tail_1 p5}
 +
 \frac{c_2}{j}
 \int_0^t
 \mathrm{E}
 \left (
 \|\mathbf{u}^{ \varepsilon, \delta} (s) \|_{H^1(\R^2) }^2
 \right ) ds+\frac{c_1}{j}\int_0^t\mathrm{E}
  \left(\delta\|\mathbf{ u}^{ \varepsilon, \delta}
(s)\|^2_{H^{3}
(\R^2) }\right) ds.
 \end{equation}
By \eqref{tail_1 p5} and \eqref{pro2 1}-\eqref{pro2 2} with $p=1$,  we infer that  there exists
$c_4 = c_4 (T, \mathbf{u}_0)>0$ such that
for all $t\in [0,T]$,
  $ j \in \mathbb{N}$, $\varepsilon\in [0,1]$
  and  $\delta\in (0,1]$,
$$
 \mathrm{E}
 \left (\sup_{0\le r \le t}
  \|\phi_j \mathbf{u}^{ \varepsilon, \delta} (r)
 \|_{L^2(\R^2)
 }^2
 \right )
 \leq
 \mathrm{E}
 \left (
 \|\phi_j \mathbf{u}_0
 \|_{L^2(\R^2)
 }^2
 \right ) +
 \frac{c_2c_4}{j}
 +\frac{c_1c_4}{j}
 $$
 \begin{equation}\label{tail_1 p6}
 +\left (c_3
 +
  3\sum_{k=1}^\infty\|\mathbf{f}_k\|^2_{L^\infty
(\mathbb{R} ^2) }
\right )
\int_0^t
 \mathrm{E}
 \left (
 \|\phi_j
\mathbf{u}^{ \varepsilon, \delta} (s) \|_{L^2
(\R^2)
}^2
\right ) ds + (2T + c_3) \sum_{k=1}^\infty\|\phi_j \mathbf{f}_k\|
_{L^2(\mathbb{R}^2)  }^2.
 \end{equation}
By \eqref{tail_1 p6}
  and Gronwall's lemma we obtain that
  for all $t\in [0,T]$
and $j \in \mathbb{N}$, $\varepsilon\in [0,1]$
and  $\delta\in (0,1]$,
\be\label{tail_1 p7}
 \mathrm{E}
 \left (\sup_{0\le r\le t}
 \|\phi_j \mathbf{u}^{ \varepsilon, \delta} (r)
 \|_{L^2(\R^2)
 }^2
 \right )
 \leq
 \left (
 \|\phi_j \mathbf{u}_0
 \|_{L^2(\mathbb{R}^2 )
 }^2
 + (2T+ c_3)  \sum_{k=1}^\infty\|\phi_j \mathbf{f}_k\|
_{L^2(\mathbb{R}^2)  }^2
 +
 \frac{(c_1+c_2)c_4}{j}
 \right )
 e^{ c_5 T
 },
 \ee
where $c_5
=   c_3 + 3\sum_{k=1}^\infty\|\mathbf{f}_k\|^2_{L^\infty
(\mathbb{R} ^2) }
 $. Note that the right-hand side term of \eqref{tail_1 p7} goes to zero as $j\rightarrow\infty$, which implies that for every $\epsilon'>0$,
there exists $J=J(\epsilon', T, \mathbf{u}_0)\ge 1$
such that
for all $j\ge J$,  
 $\varepsilon\in [0,1]$, $\delta\in (0,1]$ and $t\in [0,T]$,
$$
 \mathrm{E}
 \left (\sup_{0\le r\le t}
 \|\phi_j
  \mathbf{u}^{ \varepsilon, \delta
 } (r, \mathbf{u}_0 )
 \|_{L^2(\R^2)
 }^2
 \right )
 <\epsilon',
$$
which completes the proof.
\end{proof}

Next, we discuss the tightness
of the distributions of the family
$\{
\mathbf{u}^{ \varepsilon, \delta
 }\}_{0<\delta\le 1}$
of solutions of the
viscous equation
\eqref{1.3*} for every fixed
 $\varepsilon \in [0,1]$.

\begin{lemma} \label{tight1}
If \eqref{1.1} holds, 
$ \mathbf{u}_0 \in H^1(\R^2)$
and   $\varepsilon\in [0,1]$,
then
 the family
of
distributions of
$
\{ (
\mathbf{u}^{ 
\varepsilon, \delta
 }, W) \}_{0<\delta\le 1}$ 
 is tight in the   space  $(C([0,T]; L^2 (\mathbb{R}^2))\cap  L^2(0,T;H^1 (\mathbb{R}^2)))\times C([0,T]; \mathbb{R}^\infty)$.
\end{lemma}

\begin{proof}  For every  $R>0$,
 by Chebyshev's
 inequality and \eqref{pro2 1}-\eqref{pro2 3}
  we deduce that for
 every
 $T>0$, $\alpha \in (0,{\frac 12})$
 and $q\ge 2$ with $\alpha -{\frac 1q} >0$,
\begin{align}\label{tight1 p1a}
&\mathrm{P}
\left(
\sup_{t\in [0,T]}\|
 \mathbf{u}^{ \varepsilon, \delta}
 (t)
  \|^2_{H^1(\mathbb{R}^2) }
  +\int_{0}^{T}\|  \mathbf{u}^{ \varepsilon, \delta} (t)
  \|^2_{H^2(\mathbb{R}^2)}dt
  +\|
  \mathbf{u}^{ \varepsilon, \delta}
  \|_{W^{\alpha, q}
  	(0,T; H^{-1} (\mathbb{R}^2 ))}
> R\right)\nonumber\\
&\leq \frac{\mathrm{E}
\left(
\sup_{t\in [0,T]}\|
 \mathbf{u}^{ \varepsilon, \delta}
 (t)
  \|^2_{H^1(\mathbb{R}^2) }
  +\int_{0}^{T}\|  \mathbf{u}^{ \varepsilon, \delta} (t)
  \|^2_{H^2(\mathbb{R}^2)}dt
  +\|
  \mathbf{u}^{ \varepsilon, \delta}
  \|_{W^{\alpha, q}
  	(0,T; H^{-1} (\mathbb{R}^2 ))}  
 \right)}{R}
  \nonumber\\
&\leq \frac{c_1}{R} ,
\end{align}
where $c_1=c_1 (T, \mathbf{u}_0, \alpha, q)>0$ is
a number  independent of $R,   \delta
$ and  $ \varepsilon$.
Therefore, for every
$\epsilon'>0$, 
by \eqref{tight1 p1a}
 we find that there exists
$R(\epsilon')>0$ such  that
\begin{align}\label{tight1 p1}
&\mathrm{P}
\left(
\sup_{t\in [0,T]}\|
 \mathbf{u}^{ \varepsilon, \delta}
 (t)
  \|^2_{H^1(\mathbb{R}^2) }
  +\int_{0}^{T}\|  \mathbf{u}^{ \varepsilon, \delta} (t)
  \|^2_{H^2(\mathbb{R}^2)}dt
  +\|
   \mathbf{u}^{ \varepsilon, \delta}
  \|_{W^{\alpha, q}
  	(0,T; H^{-1} (\mathbb{R}^2 ))}
> R(\epsilon') \right)  <
  {\frac 12} \epsilon'.
\end{align}

On the   other hand, by Lemma 
\ref{tail 6} we infer that for
every $\epsilon'>0$, $T>0$
and $j\in \mathbb{N}$,  
there 
  exists   $N_j
  =N_j(T,  \mathbf{u}_0, \epsilon', j) \in \mathbb{N}$  
  such that
  for all  $\varepsilon\in [0,1]$, $\delta\in (0,1]$,
\begin{align}\label{tight1 p2}
\mathrm{E}\left(
\sup_{t\in [0,T]}
\int_{|x|>{\frac 12} N_j}
 \left |
   \mathbf{u}^{ \varepsilon, \delta}
  (t) (x)
 \right |^2dx\right)\leq \frac{\epsilon'}{2^{2j+1}}.
\end{align}
By \eqref{tight1 p2} and Chebyshev's inequality we see
 for all $\varepsilon\in [0,1]$, $\delta\in (0,1]$,
\begin{align*}
&\mathrm{P}
\left(\sup_{t\in [0,T]} \int_{|x|> {\frac 12}
N_j}
 \left |
  \mathbf{u}^{ \varepsilon, \delta}
(t) (x)
 \right |^2dx
  > \frac{1}{2^j}\right)
  \leq \frac{\epsilon'}{2^{j+1}},
\end{align*}
and hence   for all  $\varepsilon\in [0,1]$, $\delta\in (0,1]$,
$$
\mathrm{P}\left(\bigcup_{j=1}^\infty
\left\{\sup_{t\in [0,T]}
\int_{|x|> {\frac 12}
 N_j}
 \left |
  \mathbf{u}^{ \varepsilon, \delta}
(t) (x)
 \right |^2dx
 > \frac{1}{2^j}
  \right\} \right )
 \leq \sum_{j=1}^\infty\frac{\epsilon'}{2^{j+1}}
= \frac{\epsilon'}{2},
$$
from 
which  we obtain  that
 for all   $\varepsilon\in [0,1]$ 
 and $\delta\in (0,1]$,
 \begin{align}\label{tight1 p4}
 \mathrm{P}
 \left(
 \left\{\sup_{t\in [0,T]}  \|\phi_{N_j}
  \mathbf{u}^{ \varepsilon, \delta} (t)
  \|_{L^2(\mathbb{R} ^2) }
  ^2
  \leq \frac{1}{2^j}, ~{\rm for~ all}~ j\in \mathbb{N}  \right\}\right)\geq 1-\frac{\epsilon'}{2},
\end{align}
where $\phi =1-\theta$ and $\theta$ is the
smooth function given by 
\eqref{cutoff}.

Define the sets:
\begin{align}
&\mathcal{S}_1^{\epsilon'}=\Bigg\{ \mathbf{v} : \sup_{t\in [0,T]}\|
 \mathbf{v} (t)
  \|^2_{H^1(\mathbb{R}^2) }
  +\int_{0}^{T}\| \mathbf{v} (t)
  \|^2_{H^2(\mathbb{R}^2)}dt
  +\|\mathbf{v} \|_{W^{\alpha, q}
  	(0,T; H^{-1} (\mathbb{R}^2 ))}
  \leq R(\epsilon')\Bigg\},\label{tight1 p5}\\
&\mathcal{S}_2^{\epsilon'}=\left\{
\mathbf{v} : \sup_{t\in [0,T]} \|\phi_{N_j}
 \mathbf{v} (t)\|_{L^2
(\mathbb{R}^2)
}^2\leq \frac{1}{2^j}, ~{\rm for~ all}~ j\in \mathbb{N}
 \right\}.\label{tight1 p6}
\end{align}
Let
$\mathcal{S}^{\epsilon'}=\mathcal{S}_1^{\epsilon'}\bigcap \mathcal{S}_2^{\epsilon'}.$
From \eqref{tight1 p1} and \eqref{tight1 p4}-\eqref{tight1 p6}, we see that for all  
$\varepsilon\in [0,1]$ and  $\delta\in (0,1]$,
$$
\mathrm{P}(  \mathbf{u}^{ \varepsilon, \delta}
 \in \mathcal{S}^{\epsilon'})> 1-\epsilon'.
 $$
 
 Next, we show 
  the set $\mathcal{S}^{\epsilon'}$ is precompact in
$C([0,T];L^2(\mathbb{R}^2))
\cap L^2(0,T;H^1 (\mathbb{R}^2))$,
for which  we only need to show
$\mathcal{S}^{\epsilon'}$
is sequentially precompact.

We first prove   that
 the set  $\mathcal{S}^{\epsilon'}$ 
 is precompact in
  $C([0,T];H^{-1}(\mathbb{R}^2) )$ by the
   Arzela-Ascoli theorem,  for which
    we need to
 verify  that
  $\mathcal{S}^{\epsilon'}$ is
  equicontinuous  in $H^{-1}(\mathbb{R}^2)$
  and 
 the set $
 \{\mathbf{v}(t):   \mathbf{v}(t)\in
 \mathcal{S}^{\epsilon'}
 \} $
  is precompact in $H^{-1}(\mathbb{R}^2)$
  for every $t\in [0,T]$.
  The  equicontinuity
  of
  $\mathcal{S}^{\epsilon'}$
   in $H^{-1}(\mathbb{R}^2)$
  follows from the embedding
  $ W^{\alpha, q}
  	(0,T; H^{-1} (\mathbb{R}^2 ))
\hookrightarrow
 C^\gamma([0,T];H^{-1} (\mathbb{R}^2 ))$ with $\gamma= \alpha-\frac{1}{q}>0$.

 On the  other hand, 
 by \eqref{tight1 p6} we see
that
for every  $\beta>0$, there exists $j_0
\in \mathbb{N}$ such that for all
$\mathbf{v} \in \mathcal{S}^{\epsilon'}$,
\begin{align}\label{tight1 p7}
\sup_{t\in [0,T]}
\|\phi_{N_{j_0}} \mathbf{v}  (t)\|_{L^2
(\mathbb{R}^2) }^2<  \frac{\beta^2}{4}.
\end{align}
Note that 
the set   $\mathcal{S}^
{\epsilon'}$ 
is bounded in $H^1
(\mathbb{R}^2)$.
 Since the embedding $H^1(|x|<\frac{3N_{j_0}}{4})$ into 
  $L^2 (|x|<\frac{3N_{j_0}}{4} ) $  is compact,
we find that
 the set $\{ (1-\phi_{N_{j_0}})
 \mathbf{v} (t)
 : \mathbf{v} (t)   \in \mathcal{S}^
 {\epsilon'} \}$ is precompact in
  $L^2 (|x|<\frac{3N_{j_0}}{4} ) $,
  and
 hence
  it has a finite open cover of
  radius
  ${\frac 14} \beta $ in 
   $L^2 (|x|<\frac{3N_{j_0}}{4}) $,
 which  together with \eqref{tight1 p7} implies that
 the set $\{ \mathbf{v}  (t)
 :   \mathbf{v}   \in \mathcal{S}^{\epsilon'}\}$
 has
  a finite open cover of
  radius
  $  \beta $ in   $L^2 (\R^2 ) $.
  Therefore, 
   the set is precompact in
    $L^2(\mathbb{R}^2)$
    and so is in
    $H^{-1}(\mathbb{R}^2)$.
   By the
   Arzela-Ascoli theorem, we 
   conclude that 
   the set
    $\mathcal{S}^
{\epsilon'}$ 
   is precompact in
   $ C([0,T]; H^{-1}(\mathbb{R}^2))$.

   Let
   $ \{\mathbf{v}_n \}_{n=1}^\infty  $
   be a sequence
   in   $\mathcal{S}^
{\epsilon'}$. 
By the precompactness of 
  $\mathcal{S}^
{\epsilon'}$   in
   $ C([0,T]; H^{-1}(\mathbb{R}^2))$,
     there exist
a subsequence of 
$   \{\mathbf{v}_n \}_{n=1}^\infty  $
(which is not relabeled)
and  $ \mathbf{v}\in  
   C([0,T]; H^{-1}(\mathbb{R}^2))$
   such that
\be\label{tight1 p7a}
 \mathbf{v}_n
    \to
    \mathbf{v}
    \ \text{in } \ 
  C([0,T]; H^{-1}(\mathbb{R}^2)).
  \ee
  Since 
  $ \{\mathbf{v}_n \}_{n=1}^\infty  $
  belongs to   $\mathcal{S}^
{\epsilon'}$,  we infer that
$\mathbf{v}\in
L^\infty(0,T; H^1(\R^2))
\cap  L^2(0,T; H^2(\R^2))$,
which along with \eqref{tight1 p7a}
further implies
$\mathbf{v}\in
C([0,T];  H_w^1(\R^2))$.
  
  We now prove
  the convergence \eqref{tight1 p7a}
  actually holds in 
$ 
  C([0,T]; L^2 (\mathbb{R}^2))$.
  Indeed, 
  since $ \mathbf{v}_n  
   \in \mathcal{S}^
{\epsilon'}$, 
  by \eqref{tight1 p7a} we have
  $$
 \sup_{t\in [0,T]}\|  \mathbf{v}_n
 (t)
  -
    \mathbf{v} (t)\|^2_{
      L^2 (\mathbb{R}^2)
    }
     \le  \sup_{t\in [0,T]}
   \|\mathbf{v}_n (t)
  -
    \mathbf{v} (t) \|_{  H^1 (\mathbb{R}^2)}
    \| 
    \mathbf{v}_n
  -
    \mathbf{v}\|_{ C([0,T]; H^{-1} (\mathbb{R}^2))}
  $$
$$
    \le  \left (
    \sqrt{ R  (\epsilon')}
    + 
    \sup_{t\in [0,T]}\| 
    \mathbf{v} (t) \|_{  H^1 (\mathbb{R}^2)}
    \right )
   \|  \mathbf{v}_n
  -
    \mathbf{v}\|_{ C([0,T]; H^{-1} (\mathbb{R}^2))}
    \to 0,
    $$
    and hence 
    $\mathbf{v}\in C([0,T]; L^2(\R^2))$
    and
    \be\label{tight1 p7b}
  \mathbf{v}_n
 \to 
  \mathbf{v}
  \ \ \text{in } C([0,T]; L^2(\R^2)).
  \ee
     
On the other hand,
 since $ \mathbf{v}_n  
   \in \mathcal{S}^
{\epsilon'}$,   by
\eqref{tight1 p7b}  we obtain
  $$
     \|  \mathbf{v}_n
-
    \mathbf{v}\|^2_{L^2(0,T; H^1(\R^2))}
  \le
     \|  \mathbf{v}_n
-
    \mathbf{v}\| _{C([0,T]; L^2(\R^2))}
    \|  \mathbf{v}_n 
-
    \mathbf{v} \| _{L^1(0,T; H^2(\R^2))}
    $$
   $$
     \le
     T^{\frac 12}  \|  \mathbf{v}_n
-
    \mathbf{v}\| _{C([0,T]; L^2(\R^2))}
    \|  \mathbf{v}_n 
-
    \mathbf{v} \| _{L^2(0,T; H^2(\R^2))}
$$
    \be\label{tight1 p7c}
     \le
     T^{\frac 12}
     \left(\sqrt{R(\epsilon')}
     + \|\mathbf{v} \| _{L^2(0,T; H^2(\R^2))}
     \right)
       \|  \mathbf{v}_n
-
    \mathbf{v}\| _{C([0,T]; L^2(\R^2))}
 \to 0.
\ee
 By 
     \eqref{tight1 p7b}-\eqref{tight1 p7c}
  we see  that the set
     $\mathcal{S}^
{\epsilon'}$ is precompact
  in
$C([0,T];L^2(\mathbb{R}^2))
 \cap L^2(0,T;H^1 (\mathbb{R}^2))$,
which completes the proof.
\end{proof}

\begin{lemma}
\label{sj1*}
If  \eqref{1.1} holds,
$\mathbf{u}_0 \in H^1(\R^2)$
and
 $\varepsilon\in [0,1]$,
 then for every sequence 
  $\delta_n \to 0$,
  there exist a subsequence
  (still denoted by $\{\delta_n\}_{n=1}^\infty$), 
  a   stochastic basis
 $(\widetilde{\Omega}, \widetilde{\mathcal{F}},
  \{\widetilde{\mathcal{F}}_t\}
  _{t\ge 0},  \widetilde{\mathrm{P}})$,
 random variables
 $(\widetilde{\mathbf{u}}^{\varepsilon}, \widetilde{W}
 ^\varepsilon )$
 and
  $(\widetilde{\mathbf{u}}^{ \varepsilon, n}, \widetilde{W}^{\varepsilon, n})$
   defined on $(\widetilde{\Omega}, \widetilde{\mathcal{F}},
    \{\widetilde{\mathcal{F}}_t\}
  _{t\ge 0},  \widetilde{\mathrm{P}})$ such that
   in  the space $(C([0,T];L^2
   (\mathbb{R}^2)
   ) 
   \cap L^2(0,T;H^1(\mathbb{R}^2)))
    \times C([0,T];\mathbb{R}^\infty)$:

{\rm (i)}. The law  of $(\widetilde{\mathbf{u}}
^{ \varepsilon, n}, \widetilde{W}^ {\varepsilon, n}
)$ is
the same as    $( \mathbf{u}^{
 \varepsilon, \delta_
n}, W)$ where 
$ \mathbf{u}^{ \varepsilon, \delta_
n}$ is the solution of \eqref{1.3*}.

{\rm (ii)}.   $(\widetilde{\mathbf{u}}^{ \varepsilon,
n }, \widetilde{W}^{ \varepsilon,  
n})\rightarrow (\widetilde{\mathbf{u}}^{\varepsilon}, \widetilde{W}^\varepsilon )$,
 $\widetilde{\mathrm{P}}$-almost surely, as $n\rightarrow\infty$.

{\rm (iii)}.  The random variables
$ \widetilde{\mathbf{u}}^{\varepsilon}$ and
 $\widetilde{\mathbf{u}}^{\varepsilon, n}$ share the same
 uniform estimates  with
 $ \mathbf{u}^{\varepsilon, \delta_n}$.

     In particular,
for all  $n\in \mathbb{N}$, $\varepsilon\in [0,1]$ 
 and $p\ge 1$,
\begin{align}\label{sj11}
&\widetilde{\mathrm{E}}\left(\sup_{t\in[0,T]}\|\widetilde{\mathbf{u}}^{\varepsilon, n}
(t)\|^{2p}_{H^1
(\mathbb{R}^2)
}\right)
+\widetilde{\mathrm{E}}\left(\int_{0}^{T}
(
 \| \widetilde{\mathbf{u}}^{\varepsilon, n} (t)\|^2_{H^2
(\mathbb{R}^2)
}
+
\delta_n \| \widetilde{\mathbf{u}}^{\varepsilon, n} (t)\|^2_{H^3
(\mathbb{R}^2)
}
) 
dt \right )^p
\leq M_1,
\end{align}
and
\begin{align}\label{sj12}
&\widetilde{\mathrm{E}}\left(\sup_{t\in[0,T]}\|\widetilde{\mathbf{u}}^{\varepsilon}
(t)\|^{2p}_{H^1
(\mathbb{R}^2)
}\right)
+\widetilde{\mathrm{E}}\left(\int_{0}^{T}
 \| \widetilde{\mathbf{u}}^{\varepsilon} (t)\|^2_{H^2
(\mathbb{R}^2)
}
  dt \right )^p
\leq M_1,
 \end{align}
where $M_1>0$ is the same
constant as   in
Proposition \ref{pro2}.

{\rm (iv)}.
$\widetilde{\mathbf{u}}^{\varepsilon, n}
\rightarrow \widetilde{\mathbf{u}}^\varepsilon $ in $L^2(\widetilde{\Omega}; C([0,T]; L^{2}(\mathbb{R}^2))\cap L^2(0,T;H^1
 (\mathbb{R}^2)
))$,  as $n\rightarrow\infty$.

{\rm (v)}.
$
(\widetilde{\Omega}, \widetilde{\mathcal{F}},
  \{\widetilde{\mathcal{F}}_t\}
  _{t\ge 0},  \widetilde{\mathrm{P}},
  \widetilde{\mathbf{u}}^{\varepsilon },
   \widetilde{W}^{\varepsilon }
  )$ is a martingale solution
  of \eqref{1.3}.

 \end{lemma}

 \begin{proof} 
 (i)-(ii). 
 Let $\{\delta_n\}_{n=1}^\infty$
 be an arbitrary sequence  such that
 $\delta_n \to 0$.
 Then by 
   Lemma \ref{tight1} and  Skorokhod's representation theorem,
   there exists a subsequence of 
   $\{\delta_n\}_{n=1}^\infty$ (not relabeled)
   such that    (i)-(ii) hold.
   
   (iii). 
   By (i) and Proposition \ref{pro2} we obtain
   \eqref{sj11},
   which together with  (ii) yields
   \eqref{sj12}.
   
   (iv). 
   By (ii) we have
   $\widetilde{\mathbf{u}}^{\varepsilon, n}
\rightarrow \widetilde{\mathbf{u}}^\varepsilon $
 in  $C([0,T]; L^2
 (\mathbb{R}^2))$,
  $\widetilde{\mathrm{P}}$-almost surely,
  which along with
  \eqref{sj11}  
   and Vitali's theorem 
   implies that
 \be\label{2.44}
\widetilde{\mathbf{u}}^{\varepsilon, n}
\rightarrow \widetilde{\mathbf{u}}^\varepsilon 
 ~{\rm in} ~L^2(\widetilde{\Omega}; C([0,T]; L^2(\mathbb{R}^2))).
 \ee

 Similarly,
  by (ii),
  \eqref{sj11}  
   and Vitali's theorem 
  we get 
$$
\widetilde{\mathbf{u}}^{\varepsilon, n}
\rightarrow \widetilde{\mathbf{u}}^\varepsilon 
 ~{\rm in} ~L^2(\widetilde{\Omega}; 
 L^2(0,T;  H^1 (\mathbb{R}^2))),
 $$
 which along with \eqref{2.44} 
 shows that (iv) is valid.
 
 (v).  
 By  (i), we find that
     $(\widetilde{\mathbf{u}}^{\varepsilon, n}, \widetilde{W}^{\varepsilon, n})$ satisfies \eqref{1.3*}:
\begin{align}\label{1.3**}
&d \widetilde{\mathbf{u}}^{\varepsilon, n}
=\Delta 
 \widetilde{\mathbf{u}}^{\varepsilon, n} dt
 -\delta_n
  \Delta^2
   \widetilde{\mathbf{u}}^{\varepsilon, n} dt
+  \widetilde{\mathbf{u}}^{\varepsilon, n}
\times \Delta 
 \widetilde{\mathbf{u}}^{\varepsilon, n} dt
-(1+|
\widetilde{\mathbf{u}}^{\varepsilon, n}
|^{2})
\widetilde{\mathbf{u}}^{\varepsilon, n}
 dt\nonumber\\
&  {+\frac {\varepsilon^2}{2}
   \sum_{k=1}^\infty
    (\widetilde{\mathbf{u}}^{\varepsilon, n}
      \times \mathbf{f}_k)
    \times \mathbf{f}_k
    dt
+\varepsilon\sum_{k=1}^\infty(
 \widetilde{\mathbf{u}}^{\varepsilon, n}  \times \mathbf{f}_k+\mathbf{f}_k) d \widetilde{W}^
 {\varepsilon, n}
 _k,} \ \ \text{in} \       H^{-1}(\R^2).
\end{align}

We now show that $(\widetilde
{\mathbf{u}}
^\varepsilon, \widetilde{W}^\varepsilon)$ 
is a martingale solution
of  \eqref{1.3}. 
To that end, we take arbitrary 
 $\varphi\in C_0^\infty(\mathbb{R}^2)$
 and
 $\psi\in L^\infty
 (\Omega \times [0,T])$.
 Then by \eqref{1.3**} we have
 $$
 \widetilde{\mathrm{E}}\int_{0}^{T}\psi(t)
  ( \widetilde{\mathbf{u}}^{\varepsilon, n}
    (t), \varphi )_{L^2
   ( \mathbb{R}^2) } dt
    =  \widetilde{\mathrm{E}}\int_{0}^{T}\psi(t)
    (  
    \widetilde{\mathbf{u}}^{\varepsilon, n}
    (0)    , \varphi )_{L^2
    	( \mathbb{R}^2) }dt
    	$$
    	$$
    +\widetilde{\mathrm{E}}\int_{0}^{T}\psi(t)
    \int_0^t
    (\Delta 
    \widetilde{\mathbf{u}}^{\varepsilon, n}
    (s), \varphi)_{L^2
    	( \mathbb{R}^2) } ds dt
    	+
    	\delta_n \widetilde{\mathrm{E}}\int_{0}^{T}\psi(t)  
    	 \int_0^t
  (\nabla  \Delta 
  \widetilde{\mathbf{u}}^{\varepsilon, n}
  (s),  \nabla \varphi)_{L^2(\R^2) } dsdt
    $$
    $$
    +\widetilde{\mathrm{E}}\int_{0}^{T}\psi(t)
    \int_0^t
    (\widetilde{\mathbf{u}}^{\varepsilon, n}
    (s) \times \Delta  \widetilde{\mathbf{u}}^{\varepsilon, n}
    (s), \varphi )
    _{L^2
    	( \mathbb{R}^2) }
    ds dt
    -\widetilde{\mathrm{E}}\int_{0}^{T}\psi(t)
    \int_0^t
    \left (
    (1+| \widetilde{\mathbf{u}}^{\varepsilon, n} (s) |^{2}) 
    \widetilde{\mathbf{u}}^{\varepsilon, n}
    (s),
    \varphi \right )_{L^2
    	( \mathbb{R}^2) } dsdt
    $$
$$
    +
    \frac{\varepsilon^2}{2}
    \widetilde{\mathrm{E}}\int_{0}^{T}\psi(t)
   \sum_{k=1}^\infty
    \int_0^t
    \left (
    ( \widetilde{\mathbf{u}}^{\varepsilon, n}
    (s)  \times \mathbf{f}_k)
    \times \mathbf{f}_k,
    \varphi
    \right )_{L^2
    	( \mathbb{R}^2 ) } dsdt
    $$
        \be\label{ps 4**} +
    \varepsilon
    \widetilde{\mathrm{E}}\int_{0}^{T}\psi(t)
    \sum_{k=1}^\infty
    \int_0^t
    ( \widetilde{\mathbf{u}}^{\varepsilon, n}
    (s)  \times \mathbf{f}_k
    +\mathbf{f}_k, \varphi )
    _{L^2
    	( \mathbb{R}^2) }
    d \widetilde{W}^{\varepsilon, n}_k dt.
    \ee
    For the viscous term,
    by \eqref{sj11} we have
 $$
 	\delta_n \widetilde{\mathrm{E}}\int_{0}^{T}\psi(t)  
    	 \int_0^t
  (\nabla  \Delta 
  \widetilde{\mathbf{u}}^{\varepsilon, n}
  (s),  \nabla \varphi)_{L^2(\R^2) } dsdt
  $$
  $$
\leq \delta_n  T\|\psi\|_{L^\infty
 (\Omega \times [0,T])}\|\varphi\|_{H^1(\mathbb{R}^2)}
\widetilde{\mathrm{E}}\int_{0}^{T}\left\| 
 \nabla \Delta
  \widetilde{\mathbf{u}}^{\varepsilon, n}
(t)\right\|_{L^2(\mathbb{R}^2)} dt$$
$$
\leq \delta_n^\frac{1}{2}T^\frac{3}{2}\|\psi\|_{L^\infty
 (\Omega \times [0,T])}\|\varphi\|_{H^1(\mathbb{R}^2)}
\widetilde{\mathrm{E}}\left(\int_{0}^{T}\delta_n
\left\|\nabla \Delta
 \widetilde{\mathbf{u}}^{\varepsilon, n}
(t)\right\|^2_{L^2(\mathbb{R}^2)} dt\right)^\frac{1}{2}
$$
$$
\leq c_1 \delta_n^\frac{1}{2}T^\frac{3}{2}\|\psi\|_{L^\infty
 (\Omega \times [0,T])}\|\varphi\|_{H^1(\mathbb{R}^2)},$$
where $c_1>0$ is
a constant  independent of $n$ and $ \varepsilon$, 
from which  we get
\be\label{ps 5*}
 \lim_{n\to \infty}
 	\delta_n \widetilde{\mathrm{E}}\int_{0}^{T}\psi(t)  
    	 \int_0^t
  (\nabla  \Delta 
  \widetilde{\mathbf{u}}^{\varepsilon, n}
  (s),  \nabla \varphi)_{L^2(\R^2) } dsdt
  =0.
\ee
  
We can  also  identify the limit of
all other  terms
in    \eqref{ps 4**} as
in \cite{Q1}, and
then by \eqref{ps 4**}-\eqref{ps 5*}
we obtain that for any
$\varphi\in C_0^\infty(\mathbb{R}^2)$
 and
 $\psi\in L^\infty
 (\Omega \times [0,T])$, 
 $$
 \widetilde{\mathrm{E}}\int_{0}^{T}\psi(t)
  ( \widetilde{\mathbf{u}}^{\varepsilon}
    (t), \varphi )_{L^2
   ( \mathbb{R}^2) } dt
    =  \widetilde{\mathrm{E}}\int_{0}^{T}\psi(t)
    (  
    \widetilde{\mathbf{u}}^{\varepsilon}
    (0)    , \varphi )_{L^2
    	( \mathbb{R}^2) }dt
    +\widetilde{\mathrm{E}}\int_{0}^{T}\psi(t)
    \int_0^t
    (\Delta 
    \widetilde{\mathbf{u}}^{\varepsilon }
    (s), \varphi)_{L^2
    	( \mathbb{R}^2) } ds dt
      $$
    $$
    +\widetilde{\mathrm{E}}\int_{0}^{T}\psi(t)
    \int_0^t
    (\widetilde{\mathbf{u}}^{\varepsilon }
    (s) \times \Delta  \widetilde{\mathbf{u}}^{\varepsilon }
    (s), \varphi )
    _{L^2
    	( \mathbb{R}^2) }
    ds dt
    -\widetilde{\mathrm{E}}\int_{0}^{T}\psi(t)
    \int_0^t
    \left (
    (1+| \widetilde{\mathbf{u}}^{\varepsilon } (s) |^{2}) 
    \widetilde{\mathbf{u}}^{\varepsilon }
    (s),
    \varphi \right )_{L^2
    	( \mathbb{R}^2) } dsdt
    $$
$$
    +
    \frac{\varepsilon^2}{2}
    \widetilde{\mathrm{E}}\int_{0}^{T}\psi(t)
   \sum_{k=1}^\infty
    \int_0^t
    \left (
    ( \widetilde{\mathbf{u}}^{\varepsilon }
    (s)  \times \mathbf{f}_k)
    \times \mathbf{f}_k,
    \varphi
    \right )_{L^2
    	( \mathbb{R}^2 ) } dsdt
    $$
        \be\label{ps 5**} +
    \varepsilon
    \widetilde{\mathrm{E}}\int_{0}^{T}\psi(t)
    \sum_{k=1}^\infty
    \int_0^t
    ( \widetilde{\mathbf{u}}^{\varepsilon }
    (s)  \times \mathbf{f}_k
    +\mathbf{f}_k, \varphi )
    _{L^2
    	( \mathbb{R}^2) }
    d \widetilde{W}^{\varepsilon }_k dt.
    \ee
    
    By the density of $C^\infty_0(\R^2)$
    in $L^2(\R^2)$, we infer from \eqref{ps 5**}
    that 
 $(\widetilde{\Omega}, \widetilde{\mathcal{F}},
  \{\widetilde{\mathcal{F}}_t\}
  _{t\ge 0},  \widetilde{\mathrm{P}},
  \widetilde{\mathbf{u}}^{\varepsilon },
   \widetilde{W}^{\varepsilon }
  )$ is a martingale solution
  of \eqref{1.3},
  which completes the proof.
  \end{proof}

We are now ready to 
show the main result of this section
regarding the convergence of
solutions of \eqref{1.3*}
as $\delta \to 0$.

{\bf Proof of Proposition \ref{pro3.1}}.
 Let $\{\delta_n\}_{n=1}^\infty $
 be an
 arbitrary  sequence
 such that $\delta_n \to 0$.
 Let  $\mathbf{u}^{\varepsilon, \delta_n}$
 be the solution of \eqref{1.3*}
 with initial data
 $\mathbf{u}_0 \in H^1(\R^2)$
 and $\delta$ replaced  by $\delta_n$.
 We first prove 
 \be\label{wana1}
 \lim_{n\to \infty}
 \mathbf{u}^{\varepsilon, \delta_n} 
 =
 \mathbf{u}^{\varepsilon}  
 \ \ \text{in probability  in }
 \ 
 C([0,T];L^2
   (\mathbb{R}^2)
   ) 
   \cap L^2(0,T;H^1(\mathbb{R}^2)),
   \ee
   where
   $
   \mathbf{u}^{\varepsilon} 
   $   
   is the solution of \eqref{1.3}
   with initial data $\mathbf{u}_0$.
   To that end, by 
   Gy\"{o}ngy-Krylov's theorem
   \cite[Lemma 1.1]{gyo},
   we need to show that
   for every sequence
   $\{(\mathbf{u}^{\varepsilon, \delta_{n_k}},
 \mathbf{u}^{\varepsilon, \delta_{m_k}} )
 \}_{k=1}^\infty$
 defined in the product space
 $(C([0,T];L^2
   (\mathbb{R}^2)
   ) 
   \cap L^2(0,T;H^1(\mathbb{R}^2)))^2$,
   there exists a subsequence
   converging weakly to a random
   variable supported on the diagonal
   of 
   $(C([0,T];L^2
   (\mathbb{R}^2)
   ) 
   \cap L^2(0,T;H^1(\mathbb{R}^2)))^2$.

   By the argument of Lemma
   \ref{tight1}, we find that
  the laws of the
   sequence
    $\{(\mathbf{u}^{\varepsilon, \delta_{n_k}},
 \mathbf{u}^{\varepsilon, \delta_{m_k}} )
 \}_{k=1}^\infty$
are tight in  $(C([0,T];L^2
   (\mathbb{R}^2)
   ) 
   \cap L^2(0,T;H^1(\mathbb{R}^2)))^2$, 
   and hence
   there exist a probability measure
   $\mu$ on 
 $(C([0,T];L^2
   (\mathbb{R}^2)
   ) 
   \cap L^2(0,T;H^1(\mathbb{R}^2)))^2$,
   and a subsequence (not relabeled)
   such that
the laws  of     
    $\{(\mathbf{u}^{\varepsilon, \delta_{n_k}},
 \mathbf{u}^{\varepsilon, \delta_{m_k}} )
 \}_{k=1}^\infty$
 weakly converge to $\mu$.
 It remains to show
 $\mu$ is 
 supported on the diagonal
   of 
   $(C([0,T];L^2
   (\mathbb{R}^2)
   ) 
   \cap L^2(0,T;H^1(\mathbb{R}^2)))^2$.
   
   Note that 
   the laws of the sequence
    $\{((\mathbf{u}^{\varepsilon, \delta_{n_k}},
 \mathbf{u}^{\varepsilon, \delta_{m_k}} ),
 W)
 \}_{k=1}^\infty$
 are tight in 
  $(C([0,T];L^2
   (\mathbb{R}^2)
   ) 
   \cap L^2(0,T;H^1(\mathbb{R}^2)))^2
    \times C([0,T];\mathbb{R}^\infty)$ 
    and hence it has a weakly
    convergent subsequence
    (which is not relabeled again).
Then by Skorokhod's representation theorem,
there exist  a
stochastic   basis 
 $(\widetilde{\Omega},\widetilde{\mathcal{F}}, 
 \{\widetilde{\mathcal{F}}_t\}_{t\ge 0},  \widetilde{\mathrm{P}})$,
  random variables $\left(
 \widetilde{\mathbf{u}}^{\varepsilon, k}, 
 \widetilde{\mathbf{v}}^{\varepsilon, k},  
 \widetilde{W}^k \right)$,
 $\left(
 \widetilde{\mathbf{u}}^{\varepsilon}, 
 \widetilde{\mathbf{v}}^{\varepsilon},  
 \widetilde{W} \right)$ 
 such that
 $\left(
 \widetilde{\mathbf{u}}^{\varepsilon, k}, 
 \widetilde{\mathbf{v}}^{\varepsilon, k},  
 \widetilde{W}^k \right)$ 
  converges to
  $\left(
 \widetilde{\mathbf{u}}^{\varepsilon}, 
 \widetilde{\mathbf{v}}^{\varepsilon},  
 \widetilde{W} \right)$,  
 $\widetilde{\mathrm{P}}$-almost surely,
in  
$(C([0,T];L^2
   (\mathbb{R}^2)
   ) 
   \cap L^2(0,T;H^1(\mathbb{R}^2)))^2
    \times C([0,T];\mathbb{R}^\infty)$.
    In addition,
    the law of
     $\left(
 \widetilde{\mathbf{u}}^{\varepsilon, k}, 
 \widetilde{\mathbf{v}}^{\varepsilon, k},  
 \widetilde{W}^k \right)$     is the same as
    $((\mathbf{u}^{\varepsilon, \delta_{n_k}},
 \mathbf{u}^{\varepsilon, \delta_{m_k}} ),
 W)   $
 for every $k\in \mathbb{N}$. 
 Consequently,
 $\mu$  must be 
 the law of 
  $\left(
 \widetilde{\mathbf{u}}^{\varepsilon}, 
 \widetilde{\mathbf{v}}^{\varepsilon} \right)$
 in 
 $(C([0,T];L^2
   (\mathbb{R}^2)
   ) 
   \cap L^2(0,T;H^1(\mathbb{R}^2)))^2
    $.    
    
    By the argument of Lemma
    \ref{sj1*} (v), one can verify
    that both
     $\left(
 \widetilde{\mathbf{u}}^{\varepsilon},  
 \widetilde{W} \right)$
 and
       $\left(
 \widetilde{\mathbf{v}}^{\varepsilon},  
 \widetilde{W} \right)$
 are martingale solutions
 of
 \eqref{1.3}
 defined in the  same stochastic
 basis
 $(\widetilde{\Omega},\widetilde{\mathcal{F}}, 
 \{\widetilde{\mathcal{F}}_t\}_{t\ge 0},  \widetilde{\mathrm{P}})$. 
 In addition, 
 $\widetilde{\mathbf{u}}^{\varepsilon}
 (0)
 =
  \widetilde{\mathbf{v}}^{\varepsilon} 
  (0)
  =
  {\mathbf{u}}_0$,
  $\widetilde{\mathrm{P}}$-almost surely.
  Then by the pathwise uniqueness
  property of \eqref{1.3} we infer that
  $ \widetilde{\mathbf{u}}^{\varepsilon} 
  =  \widetilde{\mathbf{v}}^{\varepsilon}$,
  $\widetilde{\mathrm{P}}$-almost surely,
  in
  $C([0,T];L^2
   (\mathbb{R}^2)
   ) 
   \cap L^2(0,T;H^1(\mathbb{R}^2))
    $.     
    Since $\mu$
    is the law of  
    $(\widetilde{\mathbf{u}}^{\varepsilon} ,
    \widetilde{\mathbf{v}}^{\varepsilon}   )$,
    we find that
    $\mu$ is supported    
    on the diagonal  of 
    $(C([0,T];L^2
   (\mathbb{R}^2)
   ) 
   \cap L^2(0,T;H^1(\mathbb{R}^2)))^2
    $.    
    Then \eqref{wana1} follows from
    Gy\"{o}ngy-Krylov's theorem.
    
        By \eqref{wana1} and the
        argument of Lemma \ref{sj1*},
        one can verify 
        ${\mathbf{u}}^{\varepsilon}    $
        is the solution of
        \eqref{1.3}
        with initial data
        ${\mathbf{u}}_0$
        defined in the original
        stochastic basis
        $({\Omega}, {\mathcal{F}}, 
 \{{\mathcal{F}}_t\}_{t\ge 0},  {\mathrm{P}})$.  
 Since $\{\delta_n\}_{n=1}^\infty$
 is an arbitrary sequence
 with $\delta_n \to 0$ and the solution
 of \eqref{1.3} is unique,
 we  then  obtain
         Proposition \ref{pro3.1}. \qquad\quad $\square$

\section{Uniform tail-ends estimates}
This section is devoted to the uniform 
tail-ends  estimates of solutions of
the stochastic equation
\eqref{1.3} in $H^1(\R^2)$.
We first derive   
the uniform 
tail-ends  estimates  
of  solutions of the viscous equation
\eqref{1.3*}, and then establish
the corresponding  uniform estimates 
for  \eqref{1.3}
by a limiting process based on
 Proposition \ref{pro3.1}.

\begin{lemma}\label{ues1a}
If \eqref{1.1}  holds,
then for every
 $\mathbf{u}_0 \in H^1(\R^2)$, 
 the solution $\mathbf{u}^{\varepsilon,\delta}
(t,  \mathbf{u}_0)$ of \eqref{1.3*}
 satisfies: for all $t\ge 0$,
  $\varepsilon\in [0,1]$ and $\delta\in (0,1]$,  
 $$
\mathrm{E}
\left (
 \|
 \mathbf{u}^{\varepsilon,\delta} (t,  \mathbf{u}_0) \|^2_{H^1(\mathbb{R}^2)}
 \right )
 +
  \int_0^t e^{s-t}  \mathrm{E}
\left ( \|
 \mathbf{u}^{\varepsilon,\delta} (s,  \mathbf{u}_0) \|^2_{H^2(\mathbb{R}^2)}
 +
 \int_{\mathbb{R}^2}
 |\mathbf{u}^{\varepsilon,\delta} (s,  \mathbf{u}_0)|^2
 |\nabla \mathbf{u}^{\varepsilon,\delta} (s,  \mathbf{u}_0)|^2  dx
 \right )
 ds$$
 $$
 +\delta\int_0^t e^{s-t}  \mathrm{E}
\left ( \|
 \mathbf{u}^{\varepsilon,\delta} (s,  \mathbf{u}_0) \|^2_{H^3(\mathbb{R}^2)}\right)ds
 \le
   M_3 +M_3 e^{-t}\| \mathbf{u}_0 \|^2_{H^1(\R^2)},
 $$
 and
 $$
   \int_0^t   \mathrm{E}
\left ( \|
 \mathbf{u}^{\varepsilon,\delta} (s,  \mathbf{u}_0) \|^2_{H^2(\mathbb{R}^2)}
 + 
   \delta   
  \|
 \mathbf{u}^{\varepsilon,\delta} (s,  \mathbf{u}_0) \|^2_{H^3(\mathbb{R}^2)}\right)ds
 \le
   M_3 t +M_3 \| \mathbf{u}_0 \|^2_{H^1(\R^2)},
  $$
where $M_3>0$ is a constant  depending only
 on $\{\mathbf{f}_k\}_{k=1}^\infty$ but not on $\varepsilon$ or $ \delta$.
\end{lemma}
\begin{proof} 
The proof is quite similar to
   Lemma 4.1 in \cite{Q1},
   and   hence omitted here.
 \end{proof}

\begin{lemma}\label{tail_l2}
If \eqref{1.1} holds,
then for every   $\epsilon'>0$
  and
 $\mathbf{u}_0
 \in H^1
 (\mathbb{R}^2  )$,
 there exists
 $J=J(\epsilon',
    \mathbf{u}_0) \ge 1$
    such that for all 
    $j\ge J$, $t\ge 0$,
  $\varepsilon\in [0,1]$
  and  $\delta\in (0,1]$,
  the solution 
 ${\mathbf{ u}}^{\varepsilon,\delta} (t,
\mathbf{u}_0 )$
of \eqref{1.3*} satisfies:
$$
 \mathrm{E}
\left (
 \int_{|x|>j}
 |
 {\mathbf{ u}}^{\varepsilon,\delta} (t,
\mathbf{u}_0 )(x) |^2 dx
\right )
  <\epsilon'.
$$
\end{lemma}
\begin{proof} 
By \eqref{tail_1 p1} we have for all $t\geq 0$,
 $$  {\frac {d}{dt}}
 \mathrm{E} \left (
  \|\phi_j\mathbf{u}^{\varepsilon, \delta}
(t)
\|_{L^2
(\mathbb{R}^2 ) }^2 \right )
  =2 \mathrm{E} \left (
   \Delta \mathbf{u}^{\varepsilon,\delta}
(t), \phi^2_j\mathbf{u}^{\varepsilon, \delta}   (t ) \right )_{L^2
(\mathbb{R}^2 ) }  
$$
$$+ 2
\mathrm{E} \left (
 \delta \nabla \Delta  \mathbf{u}^{\varepsilon, \delta}
(t),  \nabla ( \phi^2_j\mathbf{u}^{\varepsilon, \delta}
  (t )) \right 
 )_{L^2
(\mathbb{R}^2 ) }  
+2 \mathrm{E} \left (
   \mathbf{u}^{\varepsilon, \delta}
   (t)\times \Delta \mathbf{u}^{\varepsilon, \delta}
(t), \phi^2_j\mathbf{u}^{\varepsilon, \delta}  (t)
\right  ) _{L^2
(\mathbb{R}^2 ) } 
$$
$$-2 \mathrm{E} \left (
  (1+|\mathbf{u}^{\varepsilon, \delta}
   (t)|^{2})\mathbf{u}^{\varepsilon, \delta}  (t),
\phi^2 _j\mathbf{u}^{\varepsilon, \delta}  (t) 
\right )_{L^2
(\mathbb{R}^2 ) } 
 $$
\begin{align}\label{tail_1 p1*}
+
\varepsilon^2\sum_{k=1}^\infty
 \mathrm{E} \left (
\|\phi_j (\mathbf{u}^{\varepsilon,\delta} (t) \times \mathbf{f}_k+\mathbf{f}_k
) \| _{L^2
(\mathbb{R}^2 ) }^2 
+ \left(
(\mathbf{u}^{\varepsilon, \delta}
 (t) \times \mathbf{f}_k) \times \mathbf{f}_k, 
 \phi^2_j\mathbf{u}^{\varepsilon, \delta}  (t)\right)  _{L^2
(\mathbb{R}^2 )  }  \right )  .
\end{align}
After simple calculations, we find
that
\be\label{wanc1}
2 \mathrm{E} \left (
   \Delta \mathbf{u}^{\varepsilon,\delta}
(t), \phi^2_j\mathbf{u}^{\varepsilon, \delta}   (t ) \right )_{L^2
(\mathbb{R}^2 ) }  
\le
{\frac {c_1}{j}}
\mathrm{E} \left (
\|
\mathbf{u}^{\varepsilon, \delta}   (t ) \|^2_{H^1
(\mathbb{R}^2 ) }   \right ),
\ee
where $c_1>0$ is a number independent of  $j$,
$\varepsilon$ and $\delta$.
On the other hand, by the argument of 
  \eqref{tail_1 p}, we get
\begin{align}\label{wanc2}
2  \mathrm{E} \left (
  \delta \nabla \Delta \mathbf{u}^{\varepsilon, \delta}
(t),  \nabla ( \phi^2_j\mathbf{u}^{\varepsilon, \delta}   (t ) 
)
\right ) _{L^2
(\mathbb{R}^2 ) }  
\leq\frac{c_2 \delta}{j}  
 \mathrm{E} \left (
\|\mathbf{ u}^{\varepsilon, \delta}
(t)\|^2_{H^{3}
(\mathbb{R}^2) } \right ),
\end{align}
where $c_2>0$ is a number independent of  $j$,
$\varepsilon$ and $\delta$.
 In addition, we have
$$ 
\|\phi_j (\mathbf{u}^{\varepsilon,\delta} (t) \times \mathbf{f}_k+\mathbf{f}_k
) \| _{L^2
(\mathbb{R}^2 ) }^2 
+ \left(
(\mathbf{u}^{\varepsilon, \delta}
 (t) \times \mathbf{f}_k) \times \mathbf{f}_k, 
 \phi^2_j\mathbf{u}^{\varepsilon, \delta}  (t) \right) _{L^2
(\mathbb{R}^2 )  } 
=\|\phi_j  \mathbf{f}_k
  \| _{L^2
(\mathbb{R}^2 ) }^2 ,
$$
which along with  
\eqref{tail_1 p1*}-\eqref{wanc2}
implies that
$$  {\frac {d}{dt}}
 \mathrm{E} \left (
  \|\phi_j\mathbf{u}^{\varepsilon, \delta}
(t)
\|_{L^2
(\mathbb{R}^2 ) }^2 \right )
+2 \mathrm{E} \left (
  \|\phi_j\mathbf{u}^{\varepsilon, \delta}
(t)
\|_{L^2
(\mathbb{R}^2 ) }^2 \right )
$$
$$
\leq
{\frac {c_1}{j}}
\mathrm{E} \left (
\|
\mathbf{u}^{\varepsilon, \delta}   (t ) \|^2_{H^1
(\mathbb{R}^2 ) }   \right )
+ 
\frac{c_2 }{j}  
 \mathrm{E} \left (\delta
\|\mathbf{ u}^{\varepsilon, \delta}
(t)\|^2_{H^{3}
(\mathbb{R}^2) } \right )
+ 
 \sum_{k=1}^\infty
\|\phi_j \mathbf{f}_k\|
_{L^2(\mathbb{R}^2 ) }^2.
$$
Then
by Lemma \ref{ues1a} we obtain that for all $t\ge 0$,
$\varepsilon \in [0,1]$ and
$\delta \in (0,1]$,
$$
 \mathrm{E} \left (
  \|\phi_j\mathbf{u}^{\varepsilon, \delta}
(t)
\|_{L^2
(\mathbb{R}^2 ) }^2 \right )
+
\int_0^t e^{s-t}
 \mathrm{E} \left (
  \|\phi_j\mathbf{u}^{\varepsilon, \delta}
(s)
\|_{L^2
(\mathbb{R}^2 ) }^2 \right ) ds
$$
$$
\le e^{-t}
 \mathrm{E} \left (
  \|\phi_j\mathbf{u}_0
 \|_{L^2
(\mathbb{R}^2 ) }^2 \right )
+
{\frac {c_1}{j}} \int_0^t e^{s-t}
\mathrm{E} \left (
\|
\mathbf{u}^{\varepsilon, \delta}   (s ) \|^2_{H^1
(\mathbb{R}^2 ) }   \right ) ds
$$
$$
+ 
\frac{c_2 }{j}  \int_0^t e^{s-t}
 \mathrm{E} \left (\delta
\|\mathbf{ u}^{\varepsilon, \delta}
(s)\|^2_{H^{3}
(\mathbb{R}^2) } \right )ds
+ 
 \sum_{k=1}^\infty
\|\phi_j \mathbf{f}_k\|
_{L^2(\mathbb{R}^2 ) }^2
$$
\be\label{wanc3}
 \leq \|\phi_j\mathbf{u}_0
\|_{L^2
(\mathbb{R}^2 ) }^2
+ \frac{c_3}{j}+
 \sum_{k=1}^\infty
\|\phi_j \mathbf{f}_k\|
_{L^2(\mathbb{R}^2 ) }^2,
\ee
where $c_3>0$ depends on $\{\mathbf{f}_k\}_{k=1}^\infty$  and
$\| \mathbf{u}_0
\|_{H^1
(\mathbb{R}^2 ) }$
  but not on $j, \delta$ or $ \varepsilon$.
  Then letting
  $j\to \infty$  in
  \eqref{wanc3}, we obtain
   the   desired estimates. 
\end{proof}

Next, we  establish the uniform tail-ends
estimates of solutions to the viscous equation
\eqref{1.3*} 
in  $H^1
(\mathbb{R}^2 )$.

\begin{lemma}\label{tail_3}
If \eqref{1.1} holds,
then for every   $\epsilon'>0$
  and
 $\mathbf{u}_0
 \in H^1
 (\mathbb{R}^2  )$,
 there exists
 $J=J(\epsilon',
    \mathbf{u}_0) \ge 1$
    such that for all 
    $j\ge J$, $t\ge 0$,
  $\varepsilon\in [0,1]$
  and  $\delta\in (0,1]$,
  the solution 
 ${\mathbf{ u}}^{\varepsilon,\delta} (t,
\mathbf{u}_0 )$
of \eqref{1.3*} satisfies:
$$
 \mathrm{E}
\left (
 \int_{|x|>j}
 \left ( |
 {\mathbf{ u}^{\varepsilon,\delta}} (t,
\mathbf{u}_0 )(x) |^2
+
|\nabla
 {\mathbf{ u}^{\varepsilon,\delta}} (t,
\mathbf{u}_0 )(x) |^2
\right )
dx
\right )  <\epsilon'.
$$
\end{lemma}

\begin{proof}
Let $\phi$ and $\phi_j$ be the same cut-off
 functions as before.
  By \eqref{1.3*} and  It\^{o}'s formula,
     we get for all $t\ge 0$,
$$\| \phi_j \nabla\mathbf{u}^{\varepsilon, \delta}(t)\|
_{L^2
(\mathbb{R}^2) }^2-\| \phi_j \nabla\mathbf{u}_0\|
_{L^2
(\mathbb{R}^2) }^2-2\int_{0}^{t}
\left (   \delta \Delta^2 \mathbf{u}^ {\varepsilon, \delta} (s),  \text{div}
\left (\phi_j^2 \nabla\mathbf{u}^ {\varepsilon, \delta}(s)
\right )  \right )_{(H^{-1}
(\mathbb{R}^2),  H^{1}
(\mathbb{R}^2)) } ds$$
$$= -2\int_{0}^{t}
\left (    \Delta \mathbf{u}^{ \varepsilon, \delta}(s),  \text{div}
\left (\phi_j^2 \nabla\mathbf{u}^{ \varepsilon, \delta}(s)
\right )  \right )_{L^2
(\mathbb{R}^2) } ds$$
$$- 2\int_{0}^{t}
\left (  \mathbf{u}^{ \varepsilon, \delta}(s)\times \Delta \mathbf{u}^{ 
\varepsilon, \delta
}(s),
\text{div}
\left (   \phi_j^2
 \nabla\mathbf{u}^{
 \varepsilon, \delta
 }(s) \right ) \right ) _{L^2
(\mathbb{R}^2) } ds$$
 $$-2\int_{0}^{t}\left(  \phi_j \nabla\left((1+|\mathbf{u}^{
 \varepsilon, \delta
 }(s)|^{2})\mathbf{u}^{\varepsilon, \delta}(s)\right),
   \phi_j \nabla\mathbf{u}^{
   \varepsilon, \delta
   }(s)\right)_{L^2
(\mathbb{R}^2) } ds
$$
$$
 +\varepsilon^2\int_{0}^{t}\sum_{k= 1}^\infty (
 \phi_j
\nabla((\mathbf{u}^{
\varepsilon, \delta
}(s)\times \mathbf{f}_k)\times \mathbf{f}_k),   \phi_j \nabla\mathbf{u}^{
\varepsilon, \delta
}(s))_{L^2
(\mathbb{R}^2) }ds$$
$$
+2\varepsilon\int_{0}^{t}\sum_{k= 1}^\infty
(  \phi_j \nabla(\mathbf{u}^{
\varepsilon, \delta
}(s)\times \mathbf{f}_k+\mathbf{f}_k)  ,   \phi_j \nabla\mathbf{u}^{
\varepsilon, \delta
}(s))_{L^2
(\mathbb{R}^2) } dW_k
$$
\begin{align}\label{tail_2}
+\varepsilon^2\int_{0}^{t}\sum_{k= 1}^\infty \|  \phi_j \nabla(\mathbf{u}^{
\varepsilon, \delta
}(s)\times \mathbf{f}_k+\mathbf{f}_k)\|_{ _{L^2
(\mathbb{R}^2) }}^2ds.
 \end{align}
 By \eqref{tail_2} we have
    for all $t\ge 0$,
$$\| \phi_j \nabla\mathbf{u}^{\varepsilon, \delta}(t)\|
_{L^2
(\mathbb{R}^2) }^2-
e^{-t}
\| \phi_j \nabla\mathbf{u}_0\|
_{L^2
(\mathbb{R}^2) }^2
$$
$$
=
2\int_{0}^{t}
e^{s-t}
\left (   \delta \Delta^2 \mathbf{u}^ {\varepsilon, \delta} (s),  \text{div}
\left (\phi_j^2 \nabla\mathbf{u}^ {\varepsilon, \delta}(s)
\right )  \right )_{(H^{-1}
(\mathbb{R}^2),  H^{1}
(\mathbb{R}^2)) } ds$$
$$  -2\int_{0}^{t} e^{s-t}
\left (    \Delta \mathbf{u}^{ \varepsilon, \delta}(s),  \text{div}
\left (\phi_j^2 \nabla\mathbf{u}^{ \varepsilon, \delta}(s)
\right )  \right )_{L^2
(\mathbb{R}^2) } ds$$
$$- 2\int_{0}^{t}e^{s-t}
\left (  \mathbf{u}^{ \varepsilon, \delta}(s)\times \Delta \mathbf{u}^{ 
\varepsilon, \delta
}(s),
\text{div}
\left (   \phi_j^2
 \nabla\mathbf{u}^{
 \varepsilon, \delta
 }(s) \right ) \right ) _{L^2
(\mathbb{R}^2) } ds$$
 $$-2\int_{0}^{t}e^{s-t}
 \left(  \phi_j \nabla\left((1+|\mathbf{u}^{
 \varepsilon, \delta
 }(s)|^{2})\mathbf{u}^{\varepsilon, \delta}(s)\right),
   \phi_j \nabla\mathbf{u}^{
   \varepsilon, \delta
   }(s)\right)_{L^2
(\mathbb{R}^2) } ds
$$
$$
 +\varepsilon^2
  \sum_{k= 1}^\infty
 \int_{0}^{t}e^{s-t}
 (
 \phi_j
\nabla((\mathbf{u}^{
\varepsilon, \delta
}(s)\times \mathbf{f}_k)\times \mathbf{f}_k),   \phi_j \nabla\mathbf{u}^{
\varepsilon, \delta
}(s))_{L^2
(\mathbb{R}^2) }ds$$
$$
+2\varepsilon
\sum_{k= 1}^\infty
\int_{0}^{t}e^{s-t}
 (  \phi_j \nabla(\mathbf{u}^{
\varepsilon, \delta
}(s)\times \mathbf{f}_k+\mathbf{f}_k)  ,   \phi_j \nabla\mathbf{u}^{
\varepsilon, \delta
}(s))_{L^2
(\mathbb{R}^2) } dW_k
$$
\begin{align}\label{tail_2 a}
+\varepsilon^2
\sum_{k= 1}^\infty
\int_{0}^{t}e^{s-t}
 \|  \phi_j \nabla(\mathbf{u}^{
\varepsilon, \delta
}(s)\times \mathbf{f}_k+\mathbf{f}_k)\|_{L^2
(\mathbb{R}^2) }^2ds
+\int_0^t
e^{s-t} \|  \phi_j \nabla\mathbf{u}^{
\varepsilon, \delta
}(s)\| _{L^2
(\mathbb{R}^2) }^2ds.
 \end{align}

For the artificial viscosity term, we have
for $t\ge 0$,
$$
2\int_{0}^{t} e^{s-t}
\left (   \delta \Delta^2 \mathbf{u}^{ 
\varepsilon, \delta
}(s),  \text{div}
\left (\phi_j^2 \nabla\mathbf{u}^{
\varepsilon, \delta
}(s)
\right )  \right )_{(H^{-1}
(\mathbb{R}^2),  H^{1}
(\mathbb{R}^2) )}ds
$$
$$
=
{\frac 4j}
\int_{0}^{t}e^{s-t}
\left (   \delta \Delta^2 \mathbf{u}^{ 
\varepsilon, \delta
}(s),   \phi_j \nabla \phi(x/j)  \nabla\mathbf{u}^{
\varepsilon, \delta
}(s)  \right )_{(H^{-1}
(\mathbb{R}^2),  H^{1}
(\mathbb{R}^2) )}ds
$$
$$ +
2\int_{0}^{t}e^{s-t}
\left (   \delta \Delta^2 \mathbf{u}^{ 
\varepsilon, \delta
}(s),   \phi_j^2 \Delta \mathbf{u}^{
\varepsilon, \delta
}(s)   \right )_{(H^{-1}
(\mathbb{R}^2),  H^{1}
(\mathbb{R}^2) )}ds
$$
$$
=-
{\frac 4j}
\int_{0}^{t}e^{s-t}
\left (   \delta \nabla  \Delta  \mathbf{u}^{ 
\varepsilon, \delta
}(s),   \nabla ( \phi_j \nabla \phi(x/j)  \nabla\mathbf{u}^{
\varepsilon, \delta
}(s) )  \right )_{ L^2
(\mathbb{R}^2 )}ds
$$
$$ -
2\int_{0}^{t}e^{s-t}
\left (   \delta  \nabla \Delta  \mathbf{u}^{ 
\varepsilon, \delta
}(s),   \nabla ( \phi_j^2 \Delta \mathbf{u}^{
\varepsilon, \delta
}(s))    \right )_{ L^2
(\mathbb{R}^2 )}ds
$$
 \begin{align}\label{tail_2 b}
\leq \frac{c_1}{j}\int_{0}^{t}
\delta
e^{s-t}
 \|\nabla \Delta \mathbf{u}^{
\varepsilon, \delta}(s)\|^2_{L^2(\mathbb{R}^2)}
 ds
 +\frac{c_1}{j}\int_{0}^{t} e^{s-t}
 \| \mathbf{u}^{
\varepsilon, \delta}(s)\|^2_{H^2(\mathbb{R}^2)} ds,
 \end{align}
where $  c_1>0$
is a number    independent of $j$, $ \varepsilon
$ and $ \delta$.

 For   other terms on the right-hand side of \eqref{tail_2 a}, 
by the argument of   \cite[Lemma 4.5]{Q1} we find
$$
-2\int_{0}^{t}e^{s-t}
\left (    \Delta \mathbf{u}^{\varepsilon, \delta}(s),  \text{div}
\left (\phi_j^2 \nabla\mathbf{u}^{
\varepsilon, \delta
}(s)
\right )  \right )_{L^2
(\mathbb{R}^2) } ds
$$
$$- 2\int_{0}^{t}e^{s-t}
\left (  \mathbf{u}^{
\varepsilon, \delta
}(s)\times \Delta \mathbf{u}^{
\varepsilon, \delta
}(s),
\text{div}
\left (   \phi_j^2
 \nabla\mathbf{u}^{
 \varepsilon, \delta
 }(s) \right ) \right ) _{L^2
(\mathbb{R}^2) } ds
$$
$$-2\int_{0}^{t}e^{s-t}
\left(  \phi_j \nabla\left((1+|\mathbf{u}^{
\varepsilon, \delta
}(s)|^{2})\mathbf{u}^{
\varepsilon, \delta
}(s)\right),
   \phi_j \nabla\mathbf{u}^{
   \varepsilon, \delta
   }(s)\right)_{L^2
(\mathbb{R}^2) } ds
$$
$$
 +\varepsilon^2
  \sum_{k= 1}^\infty
 \int_{0}^{t}
 e^{s-t}
 (
 \phi_j
\nabla((\mathbf{u}^{
\varepsilon, \delta
}(s)\times \mathbf{f}_k)\times \mathbf{f}_k),   \phi_j \nabla\mathbf{u}^{
\varepsilon, \delta
}(s))_{L^2
(\mathbb{R}^2) }ds$$
$$+\varepsilon^2
\sum_{k= 1}^\infty
\int_{0}^{t}
e^{s-t}
 \|  \phi_j \nabla(\mathbf{u}^{
\varepsilon, \delta
}(s)\times \mathbf{f}_k+\mathbf{f}_k)\|_{L^2
(\mathbb{R}^2)}^2ds$$
$$
\leq {\frac {c_2}j}
  \int_{0}^{t}e^{s-t}
  \left (\|\mathbf{u}^{
  \varepsilon, \delta
  } (s) \|^2
 _{H^2(\mathbb{R}^2)}
 +\int_{\mathbb{R}^2}
 |\mathbf{u}^{
 \varepsilon, \delta
 } (s) | ^2
 |\nabla \mathbf{u}^{
 \varepsilon, \delta
 } (s) |^2 dx \right )ds
  + c_2
\int_{0}^{t}e^{s-t}
\|  \phi_j  \mathbf{u}^{
\varepsilon, \delta
} (s)\|^2_{L^2(\mathbb{R}^2)}ds
$$
\begin{align}\label{tail_4}
-\int_{0}^{t}e^{s-t}
 \|  \phi_j   \nabla\mathbf{u}^{
 \varepsilon, \delta
 } (s) \|_{L^2
 (\mathbb{R}^2) }^2ds +c_2
\sum_{k=1}^\infty
\left (
\| \phi_j   \mathbf{f} _k \|^2
_{L^{ 2}(\mathbb{R}^2)}
+
 \| \phi_j \nabla \mathbf{f} _k \|^2
_{L^{ 2}(\mathbb{R}^2)}
\right ),
\end{align}
where $c_2>0$ depends on
$\{\mathbf{f} _k\}_{k=1}^\infty$ but not on  $j$,
$\varepsilon$ or $ \delta$.

From \eqref{tail_2}-\eqref{tail_4}, we get for all $\varepsilon\in [0,1]$, $\delta\in (0,1]$,
$$
\mathrm{E}\left(
\| \phi_j \nabla\mathbf{u}^{\varepsilon, \delta}(t)\|
_{L^2
(\mathbb{R}^2) }^2
\right )
\le 
e^{-t}
\| \phi_j \nabla\mathbf{u}_0\|
_{L^2
(\mathbb{R}^2) }^2
 +\frac{c_1}{j}\mathrm{E}
 \left (
 \int_{0}^{t}
\delta e^{s-t}\|\nabla \Delta \mathbf{u}^{
\varepsilon, \delta}(s)\|^2_{L^2(\mathbb{R}^2)}
 ds
 \right )
 $$
$$
+ {\frac {c_1+ c_2}j}
  \mathrm{E}\left (
  \int_{0}^{t}e^{s-t}\left (\|\mathbf{u}^{
  \varepsilon, \delta} (s) \|^2
 _{H^2(\mathbb{R}^2)}
 +\int_{\mathbb{R}^2}
 |\mathbf{u}^{\varepsilon,
 \delta} (s) | ^2
 |\nabla \mathbf{u}^{
 \varepsilon, \delta} (s) |^2 dx \right )ds
 \right )
$$
 \be\label{tail_4 a}
  + c_2
\mathrm{E}
\left (
\int_{0}^{t}e^{s-t}
\|  \phi_j  \mathbf{u}^{\varepsilon,
\delta} (s)\|^2_{L^2(\mathbb{R}^2)}ds
\right )
 +c_2
\sum_{k=1}^\infty
\left (
\| \phi_j   \mathbf{f} _k \|^2
_{L^{ 2}(\mathbb{R}^2)}
+
 \| \phi_j \nabla \mathbf{f} _k \|^2
_{L^{ 2}(\mathbb{R}^2)}
\right ).
\ee
By 
\eqref{wanc3},
  \eqref{tail_4 a} and
   Lemma \ref{ues1a},
 we   obtain for all 
 $t\ge 0$,
 $\varepsilon\in [0,1]$, $\delta\in (0,1]$,
$$\mathrm{E}\left(\| \phi_j \nabla\mathbf{u}^{
\varepsilon, \delta}(t)\|
_{L^2
(\mathbb{R}^2) }^2\right)
\leq
\left (\| \phi_j  \mathbf{u}_0
 \|
_{L^2
(\mathbb{R}^2) }^2 
+
 \| \phi_j \nabla\mathbf{u}_0 
  \|
_{L^2
(\mathbb{R}^2) }^2
\right )
 $$
\begin{align}\label{tail 5}
 +\frac{c_3}{j}
+ c_3
 \sum_{k=1}^\infty
\left (
\| \phi_j   \mathbf{f} _k \|^2
_{L^{ 2}(\mathbb{R}^2)}
+
 \| \phi_j \nabla \mathbf{f} _k \|^2
_{L^{ 2}(\mathbb{R}^2)}
\right ),
\end{align}
where $c_3>0$
is a number depending on
  $\{\mathbf{f} _k\}_{k=1}^\infty$
  and $\nabla\mathbf{u}_0 $,
  but not on $j$,
  $\varepsilon$ or $\delta$. 
  
By  \eqref{1.1} and
\eqref{tail 5} we find that
for every  $\epsilon'>0$, there exists $J=J(\epsilon',
    \mathbf{u}_0) \ge 1$
 such that
 for all $j\ge J$, 
 $t\ge 0$,
 $\varepsilon\in [0,1]$ and  $\delta\in (0,1]$,
  $$ \mathrm{E}\left(\| \phi_j \nabla\mathbf{u}^{
  \varepsilon, \delta}(t)\|
_{L^2
(\mathbb{R}^2) }^2\right)<\epsilon',$$
which along with Lemma \ref{tail_l2}
completes  the proof.
\end{proof}

By \eqref{tail_4 a} we  also obtain the
following uniform tail-ends estimates
of solutions
of \eqref{1.3*}
with bounded initial data
for large time.

\begin{corollary}\label{tail_3a}
If \eqref{1.1} holds, 
then for every $\epsilon'>0$
and $R>0$, there exist
$T=T(\epsilon', R)>0$
and
  $J=J(\epsilon') \ge 1$
 such that
 for all
 $t\ge T$,
 $j\ge J$, $\varepsilon\in [0,1]$
 and  $\delta\in (0,1]$,
 the solution 
$ \mathbf{ u}^{\varepsilon,\delta} (t,
\mathbf{u}_0 )$
 of
 \eqref{1.3*} with
 $\| \mathbf{u}_0\| _{H^1(\mathbb{R}^2)}\leq R$
 satisfies:  
$$
 \mathrm{E}
\left (
 \int_{|x|>j}
 \left ( |
 {\mathbf{ u}^{\varepsilon,\delta}} (t,
\mathbf{u}_0 )(x) |^2
+
|\nabla
 {\mathbf{ u}^{\varepsilon,\delta}} (t,
\mathbf{u}_0 )(x) |^2
\right )
dx
\right )  <\epsilon'.
$$
\end{corollary}

We are now in a position to present the
main result of this section on the
uniform tail-ends estimates of solutions
of \eqref{1.3}.
 
\begin{proposition}\label{pro3a}
 If \eqref{1.1} holds,
then for every   $\epsilon'>0$
  and
 $\mathbf{u}_0
 \in H^1
 (\mathbb{R}^2  )$,
 there exists
 $J=J(\epsilon',
    \mathbf{u}_0) \ge 1$
    such that for all 
    $j\ge J$, $t\ge 0$ and
  $\varepsilon\in [0,1]$,
  the solution 
 ${\mathbf{ u}}^{\varepsilon} (t,
\mathbf{u}_0 )$
of \eqref{1.3} satisfies:
$$
 \mathrm{E}
\left (
 \int_{|x|>j}
 \left ( |
 {\mathbf{ u}^{\varepsilon}} (t,
\mathbf{u}_0 )(x) |^2
+
|\nabla
 {\mathbf{ u}^{\varepsilon}} (t,
\mathbf{u}_0 )(x) |^2
\right )
dx
\right )  <\epsilon'.
$$
 \end{proposition}

\begin{proof}
Let $\varepsilon \in [0,1]$ be fixed. 
By Proposition \ref{pro3.1} we find that
there exists a sequence
 $\delta_n \to 0$ such that,
 $\mathrm{P}$-almost surely,
 $$
 \lim_{n\to \infty}
 {\mathbf{ u}}^{\varepsilon, \delta_n} (\cdot,
\mathbf{u}_0 )
= {\mathbf{ u}}^{\varepsilon } (\cdot,
\mathbf{u}_0 )
\ \text{ in } C([0, 1]; L^2(\R^2)),
$$
 where ${\mathbf{ u}}^{\varepsilon, \delta_n} (\cdot,
\mathbf{u}_0 )$
and
$
{\mathbf{ u}}^{\varepsilon} (\cdot,
\mathbf{u}_0 )$ are the solutions
of \eqref{1.3*}
and \eqref{1.3}, respectively.

By Proposition \ref{pro3.1}  again,
we see that
there exists a subsequence
$\{\delta_{n_k}\}_{k=1}^\infty$
of 
$\{\delta_{n}\}_{n=1}^\infty$
  such that,
 $\mathrm{P}$-almost surely,
$$
 \lim_{k\to \infty}
 {\mathbf{ u}}^{\varepsilon, \delta_{n_k}} (\cdot,
\mathbf{u}_0 )
= {\mathbf{ u}}^{\varepsilon } (\cdot,
\mathbf{u}_0 )
\ \text{ in } C([0, 2]; L^2(\R^2)).
$$
Repeating this process, by a diagonal method,
we find that there exist a
subsequence
of 
$\{\delta_{n}\}_{n=1}^\infty$
(which is not  relabeled), 
and a set $\Omega_0$ of full probability measure
 such that, for every $m\in \mathbb{N}$ and
 $\omega \in \Omega_0$, 
 \be\label{pro3a p0}
 \lim_{n\to \infty}
 {\mathbf{ u}}^{\varepsilon, \delta_{n}} (\cdot,
\mathbf{u}_0, \omega )
= {\mathbf{ u}}^{\varepsilon } (\cdot,
\mathbf{u}_0, \omega )
\ \text{ in } C([0, m]; L^2(\R^2)).
\ee
By \eqref{pro3a p0} we infer that
for every
$t\ge 0$ and 
$\omega \in \Omega_0$, 
 \be\label{pro3a p0a}
 \lim_{n\to \infty}
 {\mathbf{ u}}^{\varepsilon, \delta_{n}} (t,
\mathbf{u}_0, \omega )
= {\mathbf{ u}}^{\varepsilon } (t,
\mathbf{u}_0, \omega )
\ \text{ in }  L^2(\R^2).
\ee

Given $j\in \mathbb{N}$, let $\mathcal{O}_j
=\{ x\in \mathbb{R}^2: |x|>j\}$.
By \eqref{pro3a p0a} we get, 
 for every $j\in \mathbb{N}$, 
$t\ge 0$ and 
$\omega \in \Omega_0$, 
 \be\label{pro3a p1}
 \lim_{n\to \infty}
 {\mathbf{ u}}^{\varepsilon, \delta_n} (t,
\mathbf{u}_0, \omega )
= {\mathbf{ u}}^{\varepsilon } (t,
\mathbf{u}_0, \omega )
\ \text{ in }  L^2(\mathcal{O}_j).
\ee
 Given  $\epsilon'>0$
  and
 $\mathbf{u}_0
 \in H^1
 (\mathbb{R}^2  )$,
 by Lemma \ref{tail_3} we infer that
 there exists
 $J=J(\epsilon',
    \mathbf{u}_0) \ge 1$
    such that for all 
    $j\ge J$, $t\ge 0$,
  $\varepsilon\in [0,1]$
  and $\delta\in (0,1]$,
   the solution 
 ${\mathbf{ u}}^{\varepsilon, \delta} (t,
\mathbf{u}_0 )$
of \eqref{1.3*} satisfies:
\be\label{pro3a p2}
 \mathrm{E}
\left ( \| {\mathbf{ u}^{\varepsilon, \delta}} (t,
\mathbf{u}_0  )\|^2_{H^1(\mathcal{O}_j)}
\right )  <\epsilon'.
\ee

 Define a map $\varphi:
L^2(\mathcal{O}_j) \to \R\bigcup \{+\infty\}$
by:
\be\label{pro3a p3}
\varphi (v) =
\left \{
\begin{array}{ll}
  \|v\|^2_{H^1(\mathcal{O}_j)}, &
\ \text{ if } \ v \in H^1(\mathcal{O}_j);\\
  +\infty, &
\ \text{ if } \ v \in L^2(\mathcal{O}_j)
\setminus H^1(\mathcal{O}_j).
\end{array}
\right.
\ee
Then $\varphi:
L^2(\mathcal{O}_j) \to \R\bigcup \{+\infty\}$
is lower semicontinuous.
It follows from Fatou's lemma and
\eqref{pro3a p1} that
$$
\mathrm{E}
\left (  
\varphi (
  {\mathbf{ u}^{\varepsilon}} (t,
\mathbf{u}_0  ) )
\right )
\le 
 \mathrm{E}
\left ( \liminf_{n\to \infty}
\varphi (
  {\mathbf{ u}^{\varepsilon, \delta_n}} (t,
\mathbf{u}_0  ) )
\right )
$$
$$
\le
 \liminf_{n\to \infty}
 \mathrm{E}
\left (
\varphi (
  {\mathbf{ u}^{\varepsilon, \delta_n}} (t,
\mathbf{u}_0  ) )
\right )
=
\liminf_{n\to \infty}
 \mathrm{E}
\left ( \|
  {\mathbf{ u}^{\varepsilon, \delta_n}} (t,
\mathbf{u}_0  ) \|^2_{H^1(\mathcal{O}_j)}
\right ),
$$
which along with \eqref{pro3a p2}
implies that
 for all 
    $j\ge J$, $t\ge 0$ and
  $\varepsilon\in [0,1]$,
\be\label{pro3a p4}
 \mathrm{E}
\left (  
\varphi (
  {\mathbf{ u}^{\varepsilon}} (t,
\mathbf{u}_0  ) )
\right )
 \le \epsilon'.
\ee
By \eqref{pro3a p3}-\eqref{pro3a p4} we get
for all 
    $j\ge J$, $t\ge 0$ and
  $\varepsilon\in [0,1]$, 
\begin{align*}
 \mathrm{E}
\left (  \| 
  {\mathbf{ u}^{\varepsilon}} (t,
\mathbf{u}_0  )  \|^2_{H^1(\mathcal{O}_j)}
\right )
 \le \epsilon',
\end{align*}
as desired.
\end{proof}

Similarly,  by Corollary \ref{tail_3a}
and the argument of Proposition
\ref{pro3a}, we obtain the following
uniform tail-ends estimates of solutions
of \eqref{1.3}.

\begin{proposition}\label{pro3}
 If \eqref{1.1} holds, 
then for every $\epsilon'>0$
and $R>0$, there exist
$T=T(\epsilon', R)>0$
and
  $J=J(\epsilon') \ge 1$
 such that
 for all
 $t\ge T$, 
 $j\ge J$ and  $\varepsilon\in [0,1]$,
 the solution 
$ \mathbf{ u}^{\varepsilon} (t,
\mathbf{u}_0 )$
 of
 \eqref{1.3} with
 $\| \mathbf{u}_0\| _{H^1(\mathbb{R}^2)}\leq R$
 satisfies:  
$$
 \mathrm{E}
\left (
 \int_{|x|>j}
 \left ( |
 {\mathbf{ u}^{\varepsilon}} (t,
\mathbf{u}_0 )(x) |^2
+
|\nabla
 {\mathbf{ u}^{\varepsilon}} (t,
\mathbf{u}_0 )(x) |^2
\right )
dx
\right )  <\epsilon'.
$$
 \end{proposition}

\section{Proof of main result}
In this section, 
we prove the main result
of the paper on the stability
of invariant measures of \eqref{1.3}
as the noise intensity approaches zero.
We first prove 
the tightness of the union  $\bigcup_{\varepsilon\in [0,1]}\mathcal{I}^\varepsilon$ where
$\mathcal{I}^\varepsilon$ is the set
of all invariant measures of \eqref{1.3}
in $H^1(\R^2)$ corresponding to $\varepsilon$.

\begin{lemma}\label{tigu1}
If \eqref{1.1} holds, then
  $\bigcup_{\varepsilon\in [0,1]}\mathcal{I}^\varepsilon$ is tight in $H^1(\mathbb{R}^2)$.
\end{lemma}

\begin{proof} Note that 
 Lemma \ref{ues1a}
 is also valid for $\delta =0$, and hence
  for every
$ \mathbf{u}_0 \in H^1(\mathbb{R}^2)$,
there exists $T_1=T_1(\mathbf{u}_0)>0$
such that for all $t\ge T_1$
and $\varepsilon\in [0,1]$,
the solution 
$\mathbf{u}^\varepsilon
 (t, \mathbf{u}_0)$
 of \eqref{1.3} satisfies:
 \be\label{tigu1 p1}
{\frac 1t}
\int_0^t
\mathrm{E}
\left (\|\mathbf{u}^\varepsilon  (s, \mathbf{u}_0) \|^2_{H^2(\mathbb{R}^2)}
\right ) ds \le c_1  ,
\ee
where $c_1>0$ is a constant  depending only
  on $\{\mathbf{f}_k\}_{k=1}^\infty$
  but not on $\mathbf{u}_0$  or
  $\varepsilon$.
  
  Then using \eqref{tigu1 p1}
  and the uniform tail-ends estimates
  given by 
   Proposition \ref{pro3},
   we can prove 
   the tightness of
   $\bigcup_{\varepsilon\in [0,1]}\mathcal{I}^\varepsilon$  in $H^1(\mathbb{R}^2)$
   by the argument of 
   \cite[Proposition 4.1]{Q1}.
   Since the details are similar and thus not repeated
   here.
\end{proof}

We next prove the
uniform  convergence in probability in $L^2(\mathbb{R}^2)$ of solutions
of \eqref{1.3}. Denote by $\mathbf{u}^{\varepsilon_0}(t, \mathbf{u}_0)$ 
the solution of
 \eqref{1.3} with $\varepsilon=\varepsilon_0$ starting from initial data $\mathbf{u}_0$.
\begin{lemma}\label{ucsp}
If \eqref{1.1}
holds, then for
every $L>0$,
$T>0$,
  $\eta>0$
  and $\varepsilon, \varepsilon_0\in [0,1]$,
  the solutions $\mathbf{u}^\varepsilon(t, \mathbf{u}_0)$ and $\mathbf{u}^{\varepsilon_0}(t, \mathbf{u}_0)$ of \eqref{1.3} satisfy:
\be\label{ucsp 1}\lim_{\varepsilon\rightarrow \varepsilon_0}\
\sup_{\|\mathbf{u}_0
\|_{H^1(\mathbb{R}^2)}
\le L}
\
\mathrm{P}\left(\sup_{t\in [0,T]}\|\mathbf{u}^\varepsilon(t, \mathbf{u}_0)-\mathbf{u}^{\varepsilon_0}(t, \mathbf{u}_0)\|_{L^2(\mathbb{R}^2) }\geq \eta\right)=0.
\ee
\end{lemma}

\begin{proof}
By Lemma 2.5 in \cite{Q1} we know that
there exists $c_1=c_1(L, T)>0$ such that
for all
$\varepsilon \in [0,1]$ and $
 \mathbf{u}_0 \in {H^1(\R^2)}$ with
 $
\|\mathbf{u}_0\|_{H^1(\R^2)}
\le L$, 
$$
  \mathrm{E}\left (
    \sup_{0\le t\le T}
    \|\mathbf{u}^{\varepsilon}
     (t, \mathbf{u}_{0})
    \|_{H^1(\mathbb{R}^2)} ^2
     +\int_0^{T}
    \| \mathbf{u}^{\varepsilon} (s,
  \mathbf{u}_{0})\|_{H^2(\mathbb{R}^2)}^2 ds\right)
 \le c_1,
  $$
  and hence by
   Chebyshev's inequality, we
  find that
  for a given $\epsilon'>0$, there exists
  $R=R(\epsilon', L, T)>0$ such that
  for all
 $\varepsilon  \in [0,1]$
 and
 $\mathbf{u}_0
 \in H^1(\mathbb{R}^2)$
 with
 $\|\mathbf{u}_0\|_{
   H^1(\mathbb{R}^2)} \le L$, 
\be\label{ucsp p0}
  \mathrm{P}
     \left (
    \sup_{0\le t\le T}
    \|\mathbf{u}^\varepsilon
     (t, \mathbf{u}_{0})
    \|_{H^1(\mathbb{R}^2)} ^2
     +\int_0^{T}
    \| \mathbf{u}^\varepsilon (s,
  \mathbf{u}_{0})\|_{H^2(\mathbb{R}^2)}^2 ds
    >R
     \right )
     \le {\frac 12} \epsilon'.
\ee
   Define the stopping
    times by
    $$
    \tau_{1,\varepsilon}
  =\inf\left\{t \ge 0:  \|\mathbf{u}
  ^\varepsilon  (t,
  \mathbf{u}_{0})
   \|^2_{H^1(\mathbb{R}^2) } +
   \int_{0}^{t}\| \mathbf{u}^\varepsilon
    (s,
  \mathbf{u}_{0})\|_{H^2(\mathbb{R}^2)}^2 ds
  >R   \right\},$$
and
 $$
    \tau_{1, \varepsilon_0}
  =\inf\left\{t \ge 0:  \|\mathbf{u}^
  {\varepsilon_0}  (t,
  \mathbf{u}_{0})
   \|^2_{H^1(\mathbb{R}^2) } +
   \int_{0}^{t}\| \mathbf{u}^{\varepsilon_0} (s,
  \mathbf{u}_{0})\|_{H^2(\mathbb{R}^2)}^2 ds
  >R     \right\}.$$
  As usual, $\inf \emptyset =+\infty$.
  Let
  $\tau_\varepsilon =
  \tau_{1, \varepsilon}
  \wedge \tau_{1, \varepsilon_0}$.
  It follows from \eqref{ucsp p0}  that
  \be\label{ucsp p0a}
  \mathrm{P}
  (\tau_\varepsilon  <T)
  \le \epsilon'.
  \ee
  
Note that 
for all $t\in [0, T]$, $\mathrm{P}$-almost surely,
$$
d
(\mathbf{u}^{\varepsilon}(t)
-\mathbf{u}^{\varepsilon_0}(t)
)
=\Delta (\mathbf{u}^{\varepsilon}(t)
-\mathbf{u}^{\varepsilon_0} (t)) dt
+ (\mathbf{u}^{\varepsilon}(t)
-\mathbf{u}^{\varepsilon_0}(t)
) \times \Delta  \mathbf{u}^{\varepsilon}(t)
 dt
 $$
 $$
 +
 \mathbf{u}^{\varepsilon_0}
 \times
 \Delta (\mathbf{u}^{\varepsilon}(t)
-\mathbf{u}^{\varepsilon_0}(t)
) dt
-(\mathbf{u}^{\varepsilon}(t)
-\mathbf{u}^{\varepsilon_0}(t)
) dt
-\left (
|\mathbf{u}^{\varepsilon}(t)|^2
 \mathbf{u}^{\varepsilon}(t)
 -
 |\mathbf{u}^{\varepsilon_0}(t)|^2
 \mathbf{u}^{\varepsilon_0}(t)
\right ) dt
 $$
$$
 {+\frac{
 \varepsilon^2
 -\varepsilon^2_0}{2}
 \sum_{k=1}^\infty
 (\mathbf{u}^{\varepsilon}(t)
 \times
 \mathbf{f}_k)\times
 \mathbf{f}_k
   dt
  +\frac{\varepsilon^2_0}{ 2}
  \sum_{k=1}^\infty
 ((\mathbf{u}^{\varepsilon}(t)
-\mathbf{u}^{\varepsilon_0}(t)
)\times  \mathbf{f}_k)\times \mathbf{f}_k
dt}
$$
   \be\label{ucsp p4}
 +
 (\varepsilon
 -\varepsilon_0)
 \sum_{k=1}^\infty
 (\mathbf{u}^{\varepsilon}(t)
 \times
 \mathbf{f}_k
 +
 \mathbf{f}_k
  )dW_k
  +
  \varepsilon_0
  \sum_{k=1}^\infty \left (
  (\mathbf{u}^{\varepsilon}(t)
-\mathbf{u}^{\varepsilon_0}(t)
)\times  \mathbf{f}_k
\right )  dW_k.
\ee
By \eqref{ucsp p4} and
It\^{o}'s formula we get
for all $t\in [0, T]$, $\mathrm{P}$-almost surely,
  $$
  d
\|\mathbf{u}^{\varepsilon}(t)
-\mathbf{u}^{\varepsilon_0}(t)
\|^2_{L^2(\mathbb{R}^2)}
+2\|\mathbf{u}^{\varepsilon}(t)
-\mathbf{u}^{\varepsilon_0} (t)\|^2_{H^1(\mathbb{R}^2)}
  $$
$$
\le 2 (\mathbf{u}^{\varepsilon_0}(t)
 \times
 \Delta (\mathbf{u}^{\varepsilon}(t)
-\mathbf{u}^{\varepsilon_0}(t))
, \mathbf{u}^{\varepsilon}(t)
-\mathbf{u}^{\varepsilon_0}(t))_{L^2(\mathbb{R}^2)} dt
$$
$$-2\left (
|\mathbf{u}^{\varepsilon}(t)|^2
 \mathbf{u}^{\varepsilon}(t)
 -
 |\mathbf{u}^{\varepsilon_0}(t)|^2
 \mathbf{u}^{\varepsilon_0}(t), \mathbf{u}^{\varepsilon}(t)
-\mathbf{u}^{\varepsilon_0}(t)\right)_{L^2(\mathbb{R}^2)}
 dt
$$
$$
 +
{(
 \varepsilon^2
 -\varepsilon^2_0)
 \sum_{k=1}^\infty
 ((\mathbf{u}^{\varepsilon}(t)
 \times
 \mathbf{f}_k)\times
 \mathbf{f}_k
   , \mathbf{u}^{\varepsilon}(t)
-\mathbf{u}^{\varepsilon_0}(t))_{L^2(\mathbb{R}^2)}  dt}
$$
$$
+ \varepsilon^2_0
  \sum_{k=1}^\infty
 (((\mathbf{u}^{\varepsilon}(t)
-\mathbf{u}^{\varepsilon_0}(t)
)\times  \mathbf{f}_k)\times \mathbf{f}_k, \mathbf{u}^{\varepsilon}(t)
-\mathbf{u}^{\varepsilon_0}(t))_{L^2(\mathbb{R}^2)}  dt
$$
$$
 +
 2(\varepsilon
 -\varepsilon_0)
 \sum_{k=1}^\infty
 ((\mathbf{u}^{\varepsilon}(t)
 \times
 \mathbf{f}_k
 +
 \mathbf{f}_k)dW_k, \mathbf{u}^{\varepsilon}(t)
-\mathbf{u}^{\varepsilon_0}(t))_{L^2(\mathbb{R}^2)}
$$
\be\label{up1}
+ 2(\varepsilon
 -\varepsilon_0)^2
 \sum_{k=1}^\infty
 \|\mathbf{u}^{\varepsilon}(t)
 \times
 \mathbf{f}_k
 +
 \mathbf{f}_k\|^2_{L^2(\mathbb{R}^2)}
 dt+  2\varepsilon^2_0
  \sum_{k=1}^\infty \left \|
  (\mathbf{u}^{\varepsilon}(t)
-\mathbf{u}^{\varepsilon_0}(t)
)\times  \mathbf{f}_k
\right \|^2_{L^2(\mathbb{R}^2)}  dt.
\ee
Using H\"{o}lder's inequality and the interpolation inequality, we see that
$$
2 (\mathbf{u}^{\varepsilon_0}(t)
 \times
 \Delta (\mathbf{u}^{\varepsilon}(t)
-\mathbf{u}^{\varepsilon_0}(t))
, \mathbf{u}^{\varepsilon}(t)
-\mathbf{u}^{\varepsilon_0}(t))_{L^2(\mathbb{R}^2)}
$$
$$
=-2 (\nabla\mathbf{u}^{\varepsilon_0}(t)
 \times
 \nabla (\mathbf{u}^{\varepsilon}(t)
-\mathbf{u}^{\varepsilon_0}(t))
, \mathbf{u}^{\varepsilon}(t)
-\mathbf{u}^{\varepsilon_0}(t))_{L^2(\mathbb{R}^2)}
$$
$$
\leq 2\|\nabla (\mathbf{u}^{\varepsilon}(t)
-\mathbf{u}^{\varepsilon_0}(t))\|_{L^2(\mathbb{R}^2)}\|\mathbf{u}^{\varepsilon}(t)
-\mathbf{u}^{\varepsilon_0}(t)\|_{L^4(\mathbb{R}^2)}\|\nabla\mathbf{u}^{\varepsilon_0}(t)\|_{L^4(\mathbb{R}^2)}
$$
$$
\leq c_2 \|\nabla (\mathbf{u}^{\varepsilon}(t)
-\mathbf{u}^{\varepsilon_0}(t))\|^\frac{3}{2}_{L^2(\mathbb{R}^2)}\|\mathbf{u}^{\varepsilon}(t)
-\mathbf{u}^{\varepsilon_0}(t)\|^\frac{1}{2}_{L^2(\mathbb{R}^2)}\|\nabla\mathbf{u}^{\varepsilon_0}(t)\|^\frac{1}{2}_{L^2(\mathbb{R}^2)}
\|\mathbf{u}^{\varepsilon_0}(t)\|^\frac{1}{2}_{H^2(\mathbb{R}^2)}
$$
\begin{align}\label{5.5}
\leq \|\nabla (\mathbf{u}^{\varepsilon}(t)
-\mathbf{u}^{\varepsilon_0}(t))\|^2_{L^2(\mathbb{R}^2)}+\frac{ 27c_2^4}{ 256}
\|\mathbf{u}^{\varepsilon}(t)
-\mathbf{u}^{\varepsilon_0}(t)\|^2_{L^2(\mathbb{R}^2)}\|\nabla\mathbf{u}^{\varepsilon_0}(t)\|^2_{L^2(\mathbb{R}^2)}
\|\mathbf{u}^{\varepsilon_0}(t)\|^2_{H^2(\mathbb{R}^2)}.
\end{align}
Note that
\begin{align}\label{5.6} -2\left (
|\mathbf{u}^{\varepsilon}(t)|^2
 \mathbf{u}^{\varepsilon}(t)
 -
 |\mathbf{u}^{\varepsilon_0}(t)|^2
 \mathbf{u}^{\varepsilon_0}(t), \mathbf{u}^{\varepsilon}(t)
-\mathbf{u}^{\varepsilon_0}(t)\right)_{L^2(\mathbb{R}^2)}\leq 0.
\end{align}
Moreover, we have
$$
(
 \varepsilon^2
 -\varepsilon^2_0)
 \sum_{k=1}^\infty
 ((\mathbf{u}^{\varepsilon}(t)
 \times
 \mathbf{f}_k)\times
 \mathbf{f}_k
   , \mathbf{u}^{\varepsilon}(t)
-\mathbf{u}^{\varepsilon_0}(t))_{L^2(\mathbb{R}^2)}
$$
$$
\leq  2 |\varepsilon
 -\varepsilon_0|
 \sum_{k=1}^\infty
  \|(\mathbf{u}^{\varepsilon}(t)
 \times
 \mathbf{f}_k)\times
 \mathbf{f}_k\|_{L^2(\mathbb{R}^2)} \|\mathbf{u}^{\varepsilon}(t)
-\mathbf{u}^{\varepsilon_0}(t)\|_{L^2(\mathbb{R}^2)}
$$
$$
\leq  2 | \varepsilon
 -\varepsilon_0|
  \|\mathbf{u}^{\varepsilon}(t)
-\mathbf{u}^{\varepsilon_0}(t)\|_{L^2(\mathbb{R}^2)}\|\mathbf{u}^{\varepsilon}(t)\|_{L^2(\mathbb{R}^2)}
 \sum_{k=1}^\infty \|\mathbf{f}_k\|^2_{L^\infty(\mathbb{R}^2)}
$$
\begin{align}
\leq \|\mathbf{u}^{\varepsilon}(t)
-\mathbf{u}^{\varepsilon_0}(t)\|_{L^2(\mathbb{R}^2)}^2+ {(\varepsilon
 -\varepsilon_0)^2}{}\|\mathbf{u}^{\varepsilon}(t)\|^2_{L^2(\mathbb{R}^2)}
 \left(\sum_{k=1}^\infty \|\mathbf{f}_k\|^2_{L^\infty(\mathbb{R}^2)}\right)^2,
\end{align}
and
$$
\varepsilon^2_0
  \sum_{k=1}^\infty
 (((\mathbf{u}^{\varepsilon}(t)
-\mathbf{u}^{\varepsilon_0}(t)
)\times  \mathbf{f}_k)\times \mathbf{f}_k, \mathbf{u}^{\varepsilon}(t)
-\mathbf{u}^{\varepsilon_0}(t))_{L^2(\mathbb{R}^2)}
$$
\begin{align}
\leq \varepsilon^2_0
\|\mathbf{u}^{\varepsilon}(t)
-\mathbf{u}^{\varepsilon_0}(t)\|^2_{L^2(\mathbb{R}^2)}
 \sum_{k=1}^\infty \|\mathbf{f}_k\|^2_{L^\infty(\mathbb{R}^2)}.
\end{align}
For the corrector terms, we see
$$
2(\varepsilon
 -\varepsilon_0)^2
 \sum_{k=1}^\infty
 \|\mathbf{u}^{\varepsilon}(t)
 \times
 \mathbf{f}_k
 +
 \mathbf{f}_k\|^2_{L^2(\mathbb{R}^2)}
 +2\varepsilon^2_0
  \sum_{k=1}^\infty \left \|
  (\mathbf{u}^{\varepsilon}(t)
-\mathbf{u}^{\varepsilon_0}(t)
)\times  \mathbf{f}_k
\right \|^2_{L^2(\mathbb{R}^2)}
$$
$$
\leq  4 (\varepsilon
 -\varepsilon_0)^2\sum_{k=1}^\infty
 \|\mathbf{u}^{\varepsilon}(t)\|^2_{L^2(\mathbb{R}^2)}
 \|\mathbf{f}_k\|^2_{L^\infty(\mathbb{R}^2)}
 + 4
 (\varepsilon
 -\varepsilon_0)^2\sum_{k=1}^\infty\|\mathbf{f}_k\|^2_{L^2(\mathbb{R}^2)}$$
 \begin{align}\label{up4}+2\varepsilon^2_0 \|
  \mathbf{u}^{\varepsilon}(t)
-\mathbf{u}^{\varepsilon_0}(t)\|^2_{L^2(\mathbb{R}^2)}
  \sum_{k=1}^\infty
\|  \mathbf{f}_k
 \|^2_{L^\infty(\mathbb{R}^2)}.
\end{align}
From \eqref{up1}-\eqref{up4}, we further obtain for all $t\in [0,T]$,
 $\varepsilon, \varepsilon_0 \in [0,1]$, $\mathrm{P}$-almost surely,
$$\sup_{r\in [0,t]}
\left (
\|\mathbf{u}^{\varepsilon}(r\wedge \tau_\varepsilon)
-\mathbf{u}^{\varepsilon_0}(r\wedge \tau_\varepsilon)
\|^2_{L^2(\mathbb{R}^2)}
+\int_{0}^{r\wedge \tau_\varepsilon}\|\mathbf{u}^{\varepsilon}(s)
-\mathbf{u}^{\varepsilon_0} (s)\|^2_{H^1(\mathbb{R}^2)} ds
\right )
$$
$$
\leq \int_{0}^{t\wedge \tau_\varepsilon}\|\mathbf{u}^{\varepsilon}(s)
-\mathbf{u}^{\varepsilon_0}(s)\|^2_{L^2(\mathbb{R}^2)}\left(\frac{ 27c_2^4}{ 256}\|\nabla\mathbf{u}^{\varepsilon_0}(s)\|^2_{L^2(\mathbb{R}^2)}
\|\mathbf{u}^{\varepsilon_0}(s)\|^2_{H^2(\mathbb{R}^2)}+3\varepsilon^2_0
  \sum_{k=1}^\infty
\|  \mathbf{f}_k
 \|^2_{L^\infty(\mathbb{R}^2)}\right) ds
$$
$$
+ |\varepsilon-\varepsilon_0|^2
 \int_{0}^{t\wedge \tau_\varepsilon}
 \|\mathbf{u}^{\varepsilon}(s)\|^2_{L^2(\mathbb{R}^2)}\left(\sum_{k=1}^\infty \|\mathbf{f}_k\|^2_{L^2(\mathbb{R}^2)}\right)^2ds
$$
$$
+ 4|\varepsilon-\varepsilon_0|^2 \int_{0}^{t\wedge \tau_\varepsilon} \left(\|\mathbf{u}^{\varepsilon}(s)\|^2_{L^2(\mathbb{R}^2)}\sum_{k=1}^\infty
 \|\mathbf{f}_k\|^2_{L^\infty(\mathbb{R}^2)}+
\sum_{k=1}^\infty \|\mathbf{f}_k\|^2_{L^2(\mathbb{R}^2)}\right)ds
$$
$$
 +2|\varepsilon
 -\varepsilon_0|
 \sup_{r\in [0,t]}\left|\int_{0}^{r\wedge \tau_\varepsilon}\sum_{k=1}^\infty
 ((\mathbf{u}^{\varepsilon}(s)
 \times
 \mathbf{f}_k
 +
 \mathbf{f}_k)dW_k, \mathbf{u}^{\varepsilon}(s)
-\mathbf{u}^{\varepsilon_0}(s))_{L^2(\mathbb{R}^2)}\right|
$$
\begin{align*}
&\leq \int_{0}^{t}\left(\sup_{r\in [0,s]}\|\mathbf{u}^{\varepsilon}(r\wedge \tau_\varepsilon)
-\mathbf{u}^{\varepsilon_0}(r\wedge \tau_\varepsilon)\|^2_{L^2(\mathbb{R}^2)}\right)1_{(0,\tau_\varepsilon)}(s)
 \nonumber\\
&\qquad\qquad\qquad \times\left(\frac{ 27c_2^4}{ 256}\|\nabla\mathbf{u}^{\varepsilon_0}(s)\|^2_{L^2(\mathbb{R}^2)}
\|\mathbf{u}^{\varepsilon_0}(s)\|^2_{H^2(\mathbb{R}^2)}+3\varepsilon^2_0
  \sum_{k=1}^\infty
\|  \mathbf{f}_k
 \|^2_{L^\infty(\mathbb{R}^2)}\right) ds
\end{align*}
$$
+{|\varepsilon-\varepsilon_0|^2}{}\int_{0}^{t\wedge \tau_\varepsilon}
 \|\mathbf{u}^{\varepsilon}(s)\|^2_{L^2(\mathbb{R}^2)}\left(\sum_{k=1}^\infty \|\mathbf{f}_k\|^2_{L^2(\mathbb{R}^2)}\right)^2ds
$$
$$
+4| \varepsilon-\varepsilon_0|^2\int_{0}^{t\wedge \tau_\varepsilon} \left(\|\mathbf{u}^{\varepsilon}(s)\|^2_{L^2(\mathbb{R}^2)}\sum_{k=1}^\infty
 \|\mathbf{f}_k\|^2_{L^\infty(\mathbb{R}^2)}+
\sum_{k=1}^\infty \|\mathbf{f}_k\|^2_{L^2(\mathbb{R}^2)}\right)ds
$$
\begin{align}\label{up5}
 +2|\varepsilon
 -\varepsilon_0|
 \sup_{r\in [0,t]}\left|\int_{0}^{r\wedge \tau_\varepsilon}\sum_{k=1}^\infty
 ((\mathbf{u}^{\varepsilon}(s)
 \times
 \mathbf{f}_k
 +
 \mathbf{f}_k)dW_k, \mathbf{u}^{\varepsilon}(s)
-\mathbf{u}^{\varepsilon_0}(s))_{L^2(\mathbb{R}^2)}\right|.
\end{align}
Applying
 Gronwall's lemma
 to  \eqref{up5}, we obtain
 for all $t\in [0,T]$,
 $\varepsilon, \varepsilon_0 \in [0,1]$, $\mathrm{P}$-almost surely,
$$\sup_{r\in [0,t]}\|\mathbf{u}^{\varepsilon}(r\wedge \tau_\varepsilon)
-\mathbf{u}^{\varepsilon_0}(r\wedge \tau_\varepsilon)
\|^2_{L^2(\mathbb{R}^2)}
$$
$$
\leq  {|\varepsilon-\varepsilon_0|^2 }{}g(t)\int_{0}^{t\wedge \tau_\varepsilon}
 \|\mathbf{u}^{\varepsilon}(s)\|^2_{L^2(\mathbb{R}^2)}\left(\sum_{k=1}^\infty \|\mathbf{f}_k\|^2_{L^2(\mathbb{R}^2)}\right)^2ds
$$
$$
+ 4 |\varepsilon-\varepsilon_0|^2 g(t)\int_{0}^{t\wedge \tau_\varepsilon} \left(\|\mathbf{u}^{\varepsilon}(s)\|^2_{L^2(\mathbb{R}^2)}\sum_{k=1}^\infty
 \|\mathbf{f}_k\|^2_{L^\infty(\mathbb{R}^2)}+
\sum_{k=1}^\infty \|\mathbf{f}_k\|^2_{L^2(\mathbb{R}^2)}\right)ds
$$
\begin{align}\label{up6}
 +2 |\varepsilon
 -\varepsilon_0| g(t)
 \left(\sup_{r\in [0,t]}\left|\int_{0}^{r\wedge \tau_\varepsilon}\sum_{k=1}^\infty
 ((\mathbf{u}^{\varepsilon}(s)
 \times
 \mathbf{f}_k
 +
 \mathbf{f}_k)dW_k, \mathbf{u}^{\varepsilon}(s)
-\mathbf{u}^{\varepsilon_0}(s))_{L^2(\mathbb{R}^2)}\right|\right),
\end{align}
where $$g(t)=e^{\int_{0}^{t\wedge \tau_\varepsilon}\left(
\frac{ 27c_2^4}{ 256} \|\nabla\mathbf{u}^{\varepsilon_0}(s)\|^2_{L^2(\mathbb{R}^2)}
\|\mathbf{u}^{\varepsilon_0}(s)\|^2_{H^2(\mathbb{R}^2)}+3\varepsilon^2_0
  \sum_{k=1}^\infty
\|  \mathbf{f}_k
 \|^2_{L^\infty(\mathbb{R}^2)}\right) ds}
 $$
 $$
\leq e^{ \frac{ 27c_2^4}{ 256}R^2+3\sum_{k=1}^\infty
\|  \mathbf{f}_k
 \|^2_{L^\infty(\mathbb{R}^2)}T} .
$$
Taking expectation of \eqref{up6}, we have for all $t\in [0,T]$,  $\varepsilon, \varepsilon_0 \in [0,1]$,
$$\mathrm{E}\left(\sup_{r\in [0,t]}\|\mathbf{u}^{\varepsilon}(r\wedge \tau_\varepsilon)
-\mathbf{u}^{\varepsilon_0}(r\wedge \tau_\varepsilon)
\|^2_{L^2(\mathbb{R}^2)}\right)
$$
$$
\leq {|\varepsilon-\varepsilon_0|^2}{}e^{\frac{ 27c_2^4}{ 256} R^2+3\sum_{k=1}^\infty
\|  \mathbf{f}_k
 \|^2_{L^\infty(\mathbb{R}^2)}T}TR\left(\sum_{k=1}^\infty
\|  \mathbf{f}_k
 \|^2_{L^2(\mathbb{R}^2)}\right)^2
$$
$$
+4|\varepsilon-\varepsilon_0|^2
e^{\frac{ 27c_2^4}{ 256} R^2+3\sum_{k=1}^\infty
\|  \mathbf{f}_k
 \|^2_{L^\infty(\mathbb{R}^2)}T}\left(TR\sum_{k=1}^\infty
\|  \mathbf{f}_k
 \|_{L^\infty(\mathbb{R}^2)}^2+T\sum_{k=1}^\infty
\|  \mathbf{f}_k
 \|_{L^2(\mathbb{R}^2)}^2\right)
$$
\begin{align}\label{up7}
 &\qquad\qquad+2|\varepsilon
 -\varepsilon_0|
 e^{\frac{ 27c_2^4}{ 256} R^2+3\sum_{k=1}^\infty
\|  \mathbf{f}_k
 \|^2_{L^\infty(\mathbb{R}^2)}T}\nonumber\\
 &\qquad\qquad\quad\times \mathrm{E}
 \left(\sup_{r\in [0,t]}\left|\int_{0}^{r\wedge \tau_\varepsilon}\sum_{k=1}^\infty
 ((\mathbf{u}^{\varepsilon}(s)
 \times
 \mathbf{f}_k
 +
 \mathbf{f}_k)dW_k, \mathbf{u}^{\varepsilon}(s)
-\mathbf{u}^{\varepsilon_0}(s))_{L^2(\mathbb{R}^2)} \right|\right).
\end{align}
Let $c_3=  e^{\frac{ 27c_2^4}{ 256} R^2+3\sum_{k=1}^\infty
\|  \mathbf{f}_k
 \|^2_{L^\infty(\mathbb{R}^2)}T} $.
Using BDG's inequality, we have
$$
2|\varepsilon
 -\varepsilon_0|c_3
 \mathrm{E}
 \left(\sup_{r\in [0,t]}\left|\int_{0}^{r\wedge \tau_\varepsilon}\sum_{k=1}^\infty
 ((\mathbf{u}^{\varepsilon}(s)
 \times
 \mathbf{f}_k
 +
 \mathbf{f}_k)dW_k, \mathbf{u}^{\varepsilon}(s)
-\mathbf{u}^{\varepsilon_0}(s))_{L^2(\mathbb{R}^2)}\right|\right)
$$
$$
\leq 2c_3c_4 |\varepsilon
 -\varepsilon_0| \mathrm{E}
 \left(\int_{0}^{t\wedge \tau_\varepsilon}\sum_{k=1}^\infty
 ((\mathbf{u}^{\varepsilon}(s)
 \times
 \mathbf{f}_k
 +
 \mathbf{f}_k), \mathbf{u}^{\varepsilon}(s)
-\mathbf{u}^{\varepsilon_0}(s))^2_{L^2(\mathbb{R}^2)} ds\right)^\frac{1}{ 2}
$$
$$
\leq 2c_3c_4 | \varepsilon
 -\varepsilon_0|
 \mathrm{E}
 \left(\int_{0}^{t\wedge \tau_\varepsilon}\sum_{k=1}^\infty
 \|\mathbf{u}^{\varepsilon}(s)
 \times
 \mathbf{f}_k
 +
 \mathbf{f}_k\|^2_{L^2(\mathbb{R}^2)} \|\mathbf{u}^{\varepsilon}(s)
-\mathbf{u}^{\varepsilon_0}(s)\|^2_{L^2(\mathbb{R}^2)} ds\right)^\frac{1}{ 2}
$$
$$
\leq 2\sqrt{2}c_3c_4 | \varepsilon
 -\varepsilon_0| \mathrm{E}
 \Bigg[\left(\sup_{r\in [0,t]}\|\mathbf{u}^{\varepsilon}(r\wedge \tau_\varepsilon)
-\mathbf{u}^{\varepsilon_0}(r\wedge \tau_\varepsilon)
\|_{L^2(\mathbb{R}^2)}\right)$$
$$\times \left(\int_{0}^{t\wedge \tau_\varepsilon}\sum_{k=1}^\infty
 \left(\|\mathbf{u}^{\varepsilon}(s)\|^2_{L^2(\mathbb{R}^2)}
 \|
 \mathbf{f}_k\|^2_{L^\infty(\mathbb{R}^2)}
 +
 \|\mathbf{f}_k\|^2_{L^2(\mathbb{R}^2)}\right) ds\right)^\frac{1}{ 2}\Bigg]
$$
$$
\leq \frac{1}{2}\mathrm{E}
 \left(\sup_{r\in [0,t]}\|\mathbf{u}^{\varepsilon}(r\wedge \tau_\varepsilon)
-\mathbf{u}^{\varepsilon_0}(r\wedge \tau_\varepsilon)
\|^2_{L^2(\mathbb{R}^2)}\right)
$$
$$
+ 4c_3^2c_4^2   |\varepsilon
 -\varepsilon_0|^2\mathrm{E}\left(\int_{0}^{t\wedge \tau_\varepsilon}\sum_{k=1}^\infty
 \left(\|\mathbf{u}^{\varepsilon}(s)\|^2_{L^2(\mathbb{R}^2)}
 \|
 \mathbf{f}_k\|^2_{L^\infty(\mathbb{R}^2)}
 +
 \|\mathbf{f}_k\|^2_{L^2(\mathbb{R}^2)}\right) ds\right)
$$
$$
\leq \frac{1}{2}\mathrm{E}
 \left(\sup_{r\in [0,t]}\|\mathbf{u}^{\varepsilon}(r\wedge \tau_\varepsilon)
-\mathbf{u}^{\varepsilon_0}(r\wedge \tau_\varepsilon)
\|^2_{L^2(\mathbb{R}^2)}\right)
$$
\begin{align}\label{up8}
+ 4c_3 ^2c_4^2
|\varepsilon
 -\varepsilon_0|^2\left(TR\sum_{k=1}^\infty
 \|
 \mathbf{f}_k\|^2_{L^\infty(\mathbb{R}^2)}
 +
 T\sum_{k=1}^\infty\|\mathbf{f}_k\|^2_{L^2(\mathbb{R}^2)}\right),
\end{align}
where $c_4>0$ is a constant.
By \eqref{up7}-\eqref{up8} we have for all $t\in [0,T]$,  $\varepsilon, \varepsilon_0 \in [0,1]$,
$$
\sup_{\|
\mathbf{u}_0\|_{H^1(\R^2)}
\le L}
\mathrm{E}\left(\sup_{r\in [0,t]}\|\mathbf{u}^{\varepsilon}(r\wedge \tau_\varepsilon)
-\mathbf{u}^{\varepsilon_0}(r\wedge \tau_\varepsilon)
\|^2_{L^2(\mathbb{R}^2)}\right)
$$
$$
\leq  2 {|\varepsilon-\varepsilon_0|^2}{}
 c_3 TR\left(\sum_{k=1}^\infty
\|  \mathbf{f}_k
 \|^2_{L^2(\mathbb{R}^2)}\right)^2
$$
$$
+(8c_3+8c_3^2c_4^2)|\varepsilon-\varepsilon_0|^2
\left(TR\sum_{k=1}^\infty
\|  \mathbf{f}_k
 \|_{L^\infty(\mathbb{R}^2)}^2+T\sum_{k=1}^\infty
\|  \mathbf{f}_k
 \|_{L^2(\mathbb{R}^2)}^2\right),
$$
 and hence  
 \be\label{up9}
 \lim_{\varepsilon \to \varepsilon_0}
\sup_{\|
\mathbf{u}_0\|_{H^1(\R^2)}
\le L}
\mathrm{E}\left(\sup_{r\in [0,t]}\|\mathbf{u}^{\varepsilon}(r\wedge \tau_\varepsilon)
-\mathbf{u}^{\varepsilon_0}(r\wedge \tau_\varepsilon)
\|^2_{L^2(\mathbb{R}^2)}\right)
=0.
\ee 

 Note that  
  for all
  $\varepsilon, \varepsilon_0 \in [0,1]$,
  by \eqref{ucsp p0a} we have
 \begin{align*}
&\sup_{
\|\mathbf{u}_0\|_{H^1(\R^2)}
\le L} 
\mathrm{P}\left(\sup_{t\in [0,T]}\|\mathbf{u}^\varepsilon(t, \mathbf{u}_0)-\mathbf{u}^{\varepsilon_0}(t, \mathbf{u}_0)\|_{L^2(\mathbb{R}^2)}\geq \eta\right)\nonumber\\
&\leq
\sup_{
\|\mathbf{u}_0\|_{H^1(\R^2)}
\le L} 
\mathrm{P}\left(\left\{\sup_{t\in [0,T]}\|\mathbf{u}^\varepsilon(t, \mathbf{u}_0)-\mathbf{u}^{\varepsilon_0}(t, \mathbf{u}_0)\|_{L^2(\mathbb{R}^2)}\geq \eta\right\}\bigcap \left\{ \tau_\varepsilon \ge T\right\}\right)\nonumber\\
&\quad+
\sup_{
\|\mathbf{u}_0\|_{H^1(\R^2)}
\le L} 
\mathrm{P}\left(\left\{\sup_{t\in [0,T]}\|\mathbf{u}^\varepsilon(t, \mathbf{u}_0)-\mathbf{u}^{\varepsilon_0}(t, \mathbf{u}_0)\|_{ L^2(\mathbb{R}^2)}\geq \eta\right\}\bigcap \left\{\tau_\varepsilon
< T\right\}\right)\nonumber\\
&\leq 
\sup_{
\|\mathbf{u}_0\|_{H^1(\R^2)}
\le L} 
\mathrm{P}\left(\sup_{t\in [0,  T]}\|\mathbf{u}^\varepsilon(t
\wedge \tau_\varepsilon , \mathbf{u}_0)-\mathbf{u}^{\varepsilon_0}(t
\wedge
\tau_\varepsilon
, \mathbf{u}_0)\|_{
L^2(\mathbb{R}^2)
}\geq \eta\right)+\mathrm{P}\left( \tau_\varepsilon
<  T\right)\nonumber\\
&\leq \frac{1}{\eta^2}
\sup_{
\|\mathbf{u}_0\|_{H^1(\R^2)}
\le L} 
\mathrm{E}\left(\sup_{t\in [0, T]}\|\mathbf{u}^\varepsilon(t
\wedge \tau_\varepsilon
, \mathbf{u}_0)-\mathbf{u}^{\varepsilon_0}(t
\wedge \tau_\varepsilon , \mathbf{u}_0)\|^2_{L^2(\mathbb{R}^2)}\right)
+\epsilon'  .
\end{align*}
First taking the limit as 
  $\varepsilon\to
 \varepsilon_0$  and then
 $\epsilon' \to 0$, by \eqref{up9} we obtain
 \eqref{ucsp 1},
 which completes the proof.
\end{proof}

  In order to show that the weak limit of 
 invariant measures of \eqref{1.3} is invariant,
 we recall   the following convergence 
 of solutions  with respect to initial data
 from \cite[Lemma 2.6]{Q1}.

\begin{lemma}\label{lem5.3}
 If \eqref{1.1}
holds, and   $ \mathbf{u}^n_0\rightarrow  \mathbf{u}_0$ in $H^1(\mathbb{R}^2)$, then for
every  $\varepsilon \in [0,1]$, 
$T>0$ and $p\ge 1$, 
  the solutions $\mathbf{u}^\varepsilon(t, \mathbf{u}^n_0)$ and
   $\mathbf{u}^\varepsilon(t, \mathbf{u}_0)$ of \eqref{1.3}  satisfy:
$$\lim_{n\rightarrow \infty}\
\mathrm{E}\left(\sup_{t\in [0,T]}\|\mathbf{u}^\varepsilon(t, \mathbf{u}^n_0)-\mathbf{u}^{\varepsilon}(t, \mathbf{u}_0)\|^p
_{L^2(\mathbb{R}^2) } \right)=0.
$$
  \end{lemma}
      
We are now in a position to prove the 
main result of this paper;
 that is, Theorem \ref{th1}.

{\bf \textit{Proof of Theorem \ref{th1}}}. 
(i). 
 By Lemma \ref{tigu1} we
know    
the set $\bigcup_{\varepsilon\in [0,1]}\mathcal{I}^\varepsilon$ is tight
in $H^1(\mathbb{R}^2)$,  and hence (i) is valid.

(ii).  Suppose that $\mu$ is a probability measure
on  $H^1(\mathbb{R}^2)$,
$\varepsilon_n \to \varepsilon_0\in [0,1]$
and $\mu^{\varepsilon_n}
\in \mathcal{I}^{\varepsilon_n}$
 such that
$\mu^{\varepsilon_n}
\to \mu$ weakly.
We need to prove $\mu$ is invariant;
that is, $\mu \in \mathcal{I}^{\varepsilon_0}$.

 Let  $UC_b(L^2(\mathbb{R}^2))$
 be
  the Banach space of all bounded
uniformly continuous functions defined on $ L^2(\mathbb{R}^2)$ with the uniform norm
$\| g \|_{UC_b(L^2(\mathbb{R}^2)) }=\sup_{
v\in  L^2(\mathbb{R}^2) }
|g(v)|$.
The space $UC_b(H^1 (\mathbb{R}^2))$
is defined similarly.
Note that $UC_b(L^2(\mathbb{R}^2))$
is a subset of $UC_b(H^1 (\mathbb{R}^2))$.

Let    $g\in UC_b(L^2(\mathbb{R}^2))$ be fixed.
 Since $\{\mu^{\varepsilon_n}\} _{n=1}^\infty$ is tight in $H^1(\mathbb{R}^2)$, for every $\epsilon'>0$, 
 there exists a compact subset
 $\mathcal{K}
 =\mathcal{K}(\epsilon')$ in $ H^1(\mathbb{R}^2)$ 
   such that
\be\label{wand 1}
\mu^{\varepsilon_n}
( H^1(\mathbb{R}^2)\setminus \mathcal{K})\leq \frac{\epsilon'}{2 \| g \|_{UC_b(L^2(\mathbb{R}^2)) }  +1}.
\ee

Since 
   $g\in UC_b(L^2(\mathbb{R}^2))$, we have 
    $g\in UC_b(H^1(\mathbb{R}^2))$,
    and hence by   the invariance of 
    $\mu^{\varepsilon_n} $ in $H^1(\mathbb{R}^2)$ we obtain 
$$
 {\int_{H^1(\mathbb{R}^2)}\mathrm{E}(g(\mathbf{u}^{\varepsilon_0}(t, \mathbf{u}_0))) 
 \mu^{\varepsilon_n}
 (d\mathbf{u}_0)-\int_{H^1(\mathbb{R}^2)}g(\mathbf{u}_0)
 \mu^{\varepsilon_n}
 (d\mathbf{u}_0)}
$$
$$
=\int_{H^1(\mathbb{R}^2)}\mathrm{E}(g(\mathbf{u}^{\varepsilon_0}(t, \mathbf{u}_0)))
\mu^{\varepsilon_n}
(d\mathbf{u}_0)-\int_{H^1(\mathbb{R}^2)}\mathrm{E}(g(\mathbf{u}^{\varepsilon_n}
(t, \mathbf{u}_0)))
\mu^{\varepsilon_n}
(d\mathbf{u}_0)
$$
$$
=\int_{H^1(\mathbb{R}^2)}\mathrm{E}\big(g(\mathbf{u}^{\varepsilon_0}(t, \mathbf{u}_0))-g(\mathbf{u}^{\varepsilon_n}
(t, \mathbf{u}_0))\big)
\mu^{\varepsilon_n}
(d\mathbf{u}_0)
$$
$$
=\int_{H^1(\mathbb{R}^2)\setminus \mathcal{K}}\mathrm{E}\big(g(\mathbf{u}^{\varepsilon_0}(t, \mathbf{u}_0))-g(\mathbf{u}^
{\varepsilon_n}
(t, \mathbf{u}_0))\big)
\mu^{\varepsilon_n}
(d\mathbf{u}_0)
$$
$$
+\int_{\mathcal{K}}\mathrm{E}\big(g(\mathbf{u}^{\varepsilon_0}(t, \mathbf{u}_0))-g(\mathbf{u}^{\varepsilon_n}
(t, \mathbf{u}_0))\big)
\mu^{\varepsilon_n}
(d\mathbf{u}_0)
$$
$$
=\int_{ \mathcal{K}}\mathrm{E}\left(1_{(dist_{L^2(\mathbb{R}^2)}(\mathbf{u}^{\varepsilon_0}(t, \mathbf{u}_0), \mathbf{u}^{\varepsilon_n}
(t, \mathbf{u}_0))\geq \eta)}\big(g(\mathbf{u}^{\varepsilon_0}(t, \mathbf{u}_0))-g(\mathbf{u}^{\varepsilon_n}
(t, \mathbf{u}_0))\big)\right) \mu^{\varepsilon_n} (d\mathbf{u}_0)
$$
$$
+\int_{\mathcal{K}}\mathrm{E}\left(1_{(dist_{L^2(\mathbb{R}^2)}(\mathbf{u}^{\varepsilon_0}(t, \mathbf{u}_0), \mathbf{u}^{\varepsilon_n}
(t, \mathbf{u}_0))< \eta)}\big(g(\mathbf{u}^{\varepsilon_0}(t, \mathbf{u}_0))-g(\mathbf{u}^{\varepsilon_n}
(t, \mathbf{u}_0))\big)\right) 
\mu^{\varepsilon_n}
(d\mathbf{u}_0)
$$
\begin{align}\label{in}
+\int_{H^1(\mathbb{R}^2)\setminus \mathcal{K}}\mathrm{E}\big(g(\mathbf{u}^{\varepsilon_0}(t, \mathbf{u}_0))-g(\mathbf{u}^{
\varepsilon_n}
(t, \mathbf{u}_0))\big)
\mu^{\varepsilon_n}
(d\mathbf{u}_0).
\end{align}
 
For the third term on the right-hand side of \eqref{in}
by \eqref{wand 1}
 we have
\begin{align}\label{in0}\int_{H^1(\mathbb{R}^2)\setminus \mathcal{K}}\mathrm{E}\big(g(\mathbf{u}^{\varepsilon_0}(t, \mathbf{u}_0))
-g(\mathbf{u}^
{\varepsilon_n} (t, \mathbf{u}_0))\big)\mu^{\varepsilon_n}
(d\mathbf{u}_0)\leq \epsilon'.
\end{align}
 By Lemma \ref{ucsp} we 
 infer that  for every $\epsilon'>0$
 and  $\eta>0$, there exists $N
 =N(\epsilon', \eta) \in \mathbb{N}$ such that  
 for all $n \ge N$,
$$
\sup_{\mathbf{u}_0\in \mathcal{K}}\mathrm{P}\left(dist_{L^2(\mathbb{R}^2)}(\mathbf{u}^{\varepsilon_n}(t, \mathbf{u}_0), 
\mathbf{u}^{\varepsilon_0}
(t, \mathbf{u}_0))\geq \eta\right)\leq \frac{\epsilon'}{
2\|g\|_{UC_b(L^2(
\R^2))} +1}.
$$
Then for the first term on the right-hand side of \eqref{in} we have
$$\int_{ \mathcal{K}}\mathrm{E}\left(1_{(dist_{L^2(\mathbb{R}^2)}(\mathbf{u}^{\varepsilon_0}(t, \mathbf{u}_0), \mathbf{u}^{\varepsilon_n}
(t, \mathbf{u}_0))\geq \eta)}\big(g(\mathbf{u}^{\varepsilon_0}(t, \mathbf{u}_0))-g(\mathbf{u}^{\varepsilon_n}
(t, \mathbf{u}_0))\big)\right)\mu^{\varepsilon_n}(d\mathbf{u}_0)
$$
\begin{align}\label{in1}
\leq 2 \|g\|_{UC_b(L^2(
\R^2))} \sup_{\mathbf{u}_0\in \mathcal{K}}\mathrm{P}\left(dist_{L^2(\mathbb{R}^2)}(\mathbf{u}^{\varepsilon_0}(t, \mathbf{u}_0), \mathbf{u}^{\varepsilon_n}
(t, \mathbf{u}_0))\geq \eta\right)\leq \epsilon'.
\end{align}
Since $g\in UC_b(L^2(\mathbb{R}^2))$, for every $\epsilon'>0$ there exists $\eta>0$ such that 
$|g(v_1)-g(v_2)|\leq \epsilon'$ for all $\
v_1, v_2 \in L^2(\mathbb{R}^2)$ with
$dist_{L^2(\mathbb{R}^2)}(v_1,v_2)\leq \eta$, 
and thus   for the second term on the right-hand side of \eqref{in}  we have
\begin{align}\label{in2}
\int_{\mathcal{K}}\mathrm{E}\left(1_{(dist_{L^2(\mathbb{R}^2)}(\mathbf{u}^{\varepsilon_0}(t, \mathbf{u}_0), \mathbf{u}^{\varepsilon_n}
(t, \mathbf{u}_0))< \eta)}\big(g(\mathbf{u}^{\varepsilon_0}(t, \mathbf{u}_0))-g(\mathbf{u}^{\varepsilon_n}
(t, \mathbf{u}_0))\big)\right)\mu^{\varepsilon_n}
(d\mathbf{u}_0)\leq \epsilon'.
\end{align}
Then by  \eqref{in0}-\eqref{in2} we
find that for all
$n\ge N$,   
  \begin{align*}
\left|\int_{H^1(\mathbb{R}^2)}\mathrm{E}(g(\mathbf{u}^{\varepsilon_0}(t, \mathbf{u}_0)))\mu
^{\varepsilon_n}
(d\mathbf{u}_0)-\int_{H^1(\mathbb{R}^2)}g(\mathbf{u}_0)\mu^
{\varepsilon_n}
(d\mathbf{u}_0)\right|\leq 3\epsilon',
\end{align*} 
and hence 
  \be \label{e}
 \lim_{n\to \infty}
 \left ( \int_{H^1(\mathbb{R}^2)}
 \mathrm{E}(g(\mathbf{u}^{\varepsilon_0}
 (t, \mathbf{u}_0)))\mu
^{\varepsilon_n}
(d\mathbf{u}_0)-
\int_{H^1(\mathbb{R}^2)}g(\mathbf{u}_0)
\mu^
{\varepsilon_n}
(d\mathbf{u}_0)\right)
=0. 
\ee

If $\mathbf{u}^n_0
\to \mathbf{u}_0$ in $H^1(\R^2)$, then
{by  Lemma \ref{lem5.3} and the fact  $g\in UC_b(L^2(\mathbb{R}^2))$ we infer 
that  for every $t\geq 0$,
$$
g(\mathbf{u}^{\varepsilon_0}(t, \mathbf{u}^n_0))\rightarrow g(\mathbf{u}^{\varepsilon_0}(t, \mathbf{u}_0))  ~{\rm in ~probability},
$$
which along with the dominated convergence
theorem  implies that
$  \mathrm{E}(g(\mathbf{u}^{\varepsilon_0}(t, \mathbf{u}_0)))$ 
is a  bounded
continuous function
 with respect to
$ \mathbf{u}_0$ in $H^1(\R^2)$.
Since $\mu^{\varepsilon_n}
\to  \mu^{\varepsilon_0}$
weakly  in $H^1(\mathbb{R}^2)$, we get 
\begin{align}\label{e1}
\int_{H^1(\mathbb{R}^2)}\mathrm{E}(g(\mathbf{u}^{\varepsilon_0}(t, \mathbf{u}_0)))\mu^{\varepsilon_n}
(d\mathbf{u}_0)\rightarrow \int_{H^1(\mathbb{R}^2)}\mathrm{E}(g(\mathbf{u}^{\varepsilon_0}(t, \mathbf{u}_0)))\mu^{\varepsilon_0}(d\mathbf{u}_0),
\end{align}}
and
\begin{align}\label{e2}\int_{H^1(\mathbb{R}^2)}g(\mathbf{u}_0)\mu^
{\varepsilon_n}
(d\mathbf{u}_0)\rightarrow \int_{H^1(\mathbb{R}^2)}g(\mathbf{u}_0)\mu^{\varepsilon_0}(d\mathbf{u}_0).
\end{align}

It follows from  \eqref{e}-\eqref{e2} that
for all  $g\in UC_b(L^2(\mathbb{R}^2))$,
\be\label{e3*}
\int_{H^1(\mathbb{R}^2)}\mathrm{E}(g(\mathbf{u}^{\varepsilon_0}(t, \mathbf{u}_0)))\mu^{\varepsilon_0}(d\mathbf{u}_0)=\int_{H^1(\mathbb{R}^2)}g(\mathbf{u}_0)\mu^{\varepsilon_0}(d\mathbf{u}_0).
\ee
  
If we can prove \eqref{e3*} is valid for all
 $ {g}\in UC_b(H^1(\mathbb{R}^2))$,
 then we know
 $\mu^{\varepsilon_0}$ is an  invariant measure of \eqref{1.3} with $\varepsilon=\varepsilon_0$ in $H^1(\mathbb{R}^2)$, which will
 complete the proof.

  Given $g\in UC_b(H^1(\mathbb{R}^2))$
  and $\delta>0$,  denote by
  $$
   g_\delta
   (v) = g ((I-\delta \Delta)^{-\frac 12} v),
   \quad \forall \ v\in L^2 (\R^2).
   $$
   Then for every $\delta>0$,
   $ g_\delta \in 
  UC_b(L^2(\mathbb{R}^2))$, and hence 
  by \eqref{e3*} we get
\be\label{e4}
\int_{H^1(\mathbb{R}^2)}\mathrm{E}(g_\delta (\mathbf{u}^{\varepsilon_0}(t, \mathbf{u}_0)))\mu^{\varepsilon_0}(d\mathbf{u}_0)=\int_{H^1(\mathbb{R}^2)}g_\delta (\mathbf{u}_0)\mu^{\varepsilon_0}(d\mathbf{u}_0).
\ee
Note that if $v\in H^1(\R^2)$, then
$(I-\delta \Delta)^{-\frac 12} v
\to v$ in $H^1(\R^2)$ as $\delta \to 0$, and hence
$g_\delta v \to g(v)$,
which along with \eqref{e4}
and the dominated convergence theorem
implies that  for all  $g\in UC_b(H^1(\mathbb{R}^2))$,
$$
\int_{H^1(\mathbb{R}^2)}\mathrm{E}(g(\mathbf{u}^{\varepsilon_0}(t, \mathbf{u}_0)))\mu^{\varepsilon_0}(d\mathbf{u}_0)=\int_{H^1(\mathbb{R}^2)}g(\mathbf{u}_0)\mu^{\varepsilon_0}(d\mathbf{u}_0).
$$
This  shows 
  $\mu^{\varepsilon_0}$ is an invariant measure of  \eqref{1.3} with $\varepsilon=\varepsilon_0$ in $H^1(\mathbb{R}^2)$, and thus 
  completes the proof of Theorem \ref{th1}.

\section*{Data availability}
Data sharing is not applicable to this article as no datasets were generated or analyzed during the current
study.

\section*{Acknowledgments}
D. Huang is supported by the National Natural Science Foundation of China (Grant No. 12471228). Z. Qiu is supported by the National Natural Science Foundation of China
(Grant No. 12401305), the National Science Foundation for Colleges and Universities in
Jiangsu Province (Grant No. 24KJB110011) and the National Science Foundation of Jiangsu
Province (Grant No. BK20240721).


\section*{Statements and Declarations}
 On behalf of all authors, the corresponding author states that
there is no conflict of interest.

\end{document}